\documentclass[a4paper,10pt]{article}
\usepackage{amsmath,amssymb,amsfonts,amsthm} % 引入AMS的数学、符号、字体和定理宏包
\usepackage{mathrsfs} % 引入花体字母宏包
\usepackage{enumerate} % 用于创建自定义的列表环境
\usepackage{cases} % 用于创建多行方程中的分段函数
\usepackage{float} % 提供更多控制浮动对象的位置的选项
\usepackage{graphicx} % 引入图形宏包 Include this line if your document contains figures,
\usepackage{subfig}
\allowdisplaybreaks[4] % 允许多行公式在页面之间断开
\usepackage{wrapfig}  % 允许图形和文字环绕
\usepackage{url} % 支持URL格式
\usepackage[usenames]{color} % 引入颜色宏包，并使用预定义的颜色名称
\usepackage{setspace} % 设置行间距
\usepackage{booktabs}
\def\b1{\mbox{\boldmath $1$}} % 定义加粗的数字1

\makeatletter
\newcommand{\Biggg}{\bBigg@{3.5}} % 定义一个新的大尺寸符号
\makeatother

\theoremstyle{plain} % 设置定理样式为plain
\newtheorem{theorem}{\bf Theorem}[section] % 定义新的定理环境，编号基于章节
\newtheorem{corollary}{\bf Corollary}[section] % 定义新的推论环境，编号与定理共享
\newtheorem{lemma}{\bf Lemma}[section] % 定义新的引理环境，编号与定理共享
\newtheorem{proposition}{\bf Proposition}[section] % 定义新的命题环境，编号与定理共享
\newtheorem{definition}{\bf Definition}[section] % 定义新的定义环境，编号与定理共享
\newtheorem{remark}{\bf Remark}[section]  % 定义新的备注环境，编号与定理共享
\makeatother
\allowdisplaybreaks %  允许多行公式在页面之间断开.

\usepackage[numbers,sort&compress]{natbib} % 引入natbib宏包，使参考文献连续编号的引用以压缩方式出现
\usepackage[bookmarksnumbered, bookmarksopen,colorlinks,citecolor=blue,linkcolor=blue]{hyperref} % 引入超链接宏包，并设置超链接属性

\begin{document}
\date{}

\title{Stochastic linear-quadratic control problem for regime-switching jump-diffusion system and its application to finance}
% \thanks{Fan Wu is supported by the Natural Science Foundation of Anhui Province (Grant No. 2508085QA026) and the Provincial Natural Science Research Project of Anhui Colleges (Grant No. 2025AHGXZK30544). Xin Zhang is supported by the National Natural Science Foundation of China (Grant Nos. 12171086, 12371472), Tianyuan Fund for Mathematics grant 12426652, and Jiangsu Provincial Scientific Research Center of Applied Mathematics grant BK20233002. Xun Li is supported by the Research Grants Council of Hong Kong under grants 15221621, 15226922 and 15225124, PolyU 1-ZVXA, and 4-ZZLT.  Jie Xiong is supported by National Key R\&D Program of China (Grant No. 2022YFA1006102), and the National Natural Science Foundation of China (Grant No. 12471418).}
\author{Fan Wu \thanks{School of Big Data and Statistics, Anhui University, Hefei 230601, China; Email: wfyy121107@163.com}\quad
Xun Li \thanks{Department of Applied Mathematics, The Hong Kong Polytechnic University, Hong Kong, China; Email: li.xun@polyu.edu.hk}\quad
Jie Xiong \thanks{Department of Mathematics and SUSTech International Center for Mathematics, Southern University of Science and Technology, Shenzhen, Guangdong, 518055, China; E-mail: xiongj@sustech.edu.cn}\quad
Xin Zhang \thanks{Corresponding Author: School of Mathematics, Southeast University, Nanjing 211189, China; E-mail: x.zhang.seu@gmail.com}}
\maketitle

\noindent{\bf Abstract: }  This paper investigates a stochastic linear-quadratic (SLQ) control problem for a regime-switching jump-diffusion system. Unlike traditional regime-switching diffusion systems that couple a diffusion process with a Markov chain, we incorporate the jumps of the Markov chain into the state equation. This modeling methodology effectively captures potential \emph{gains} or \emph{losses} of the system during the regime transitions. It should be noted that the introduction of Markov chain jumps into the state equation leads to increased complexity in the corresponding coupled differential Riccati equations (CDREs), thereby rendering the solvability of the control problem more challenging. Under the assumption that the cost functional is uniformly convex, we establish the unique solvability of the corresponding CDREs. Building upon this foundation, we derive a closed-loop representation for the unique open-loop optimal control. Finally, we apply our theoretical results to the mean-variance portfolio selection problem in a regime-switching financial market with Markov chain jumps and obtain its efficient frontier.
\medskip


\section{Introduction}\label{section-1}

Let $(\Omega,\mathcal{F},\mathbb{F},\mathbb{P} )$ be a complete filtered probability space on which a one-dimensional standard Brownian motion $W(\cdot)$ and a continuous time, irreducible Markov chain $\alpha(\cdot)$ are defined with $\mathbb{F}=\{\mathcal{F}_{t}\}_{t\geq 0}$ being its natural filtration augmented by all $\mathbb{P}$-null sets in $\mathcal{F}$. Throughout this paper, we denote the state space of the Markov chain $\alpha(\cdot)$ as $\mathcal{S}:=\left\{1,2,...,L\right\}$, where $L$ is a finite natural number. The generator of the Markov chain $\alpha(\cdot)$  under $\mathbb{P}$ is defined as $\pi(t):=\left[\pi_{ij}(t)\right]_{i,j=1,2,...,L}$. Here,  for $i\neq j$,  $\pi_{ij}(t)\geq 0$ is the bounded determinate transition intensity of the chain from state $i$ to state $j$ at time $t$ and $\sum_{j=1}^{L}\pi_{ij}(t)=0$ for any fixed $i$. In the following, let $N_{j}(t)$ be the number of jumps  into state $j$ up to time $t$ and set
$$
\begin{array}{c}
\widetilde{N}_{j}(t)=N_{j}(t)-\int_{0}^{t}\lambda_{j}(s)ds\quad \text{with} \quad \lambda_{j}(s)=\pi_{\alpha(s-)j}(s)I_{\{\alpha(s-)\neq j\}}%=\sum_{i\neq j}^{L}\pi_{ij}(s).
\end{array}
$$
Then for each $j\in\mathcal{S}$, the term $\widetilde{N}_{j}(t)$ is an $\left(\mathbb{F},\mathbb{P}\right)$-martingale.

Let $\mathcal{P}$ be the $\mathbb{F}$ predictable $\sigma$-field on $[0,T]\times\Omega$, and we write $\varphi(\cdot) \in \mathcal{P}$  (respectively, $\varphi (\cdot) \in \mathbb{F}$) if the stochastic process $\varphi(\cdot)$ is $\mathcal{P}$-measurable (respectively, $\mathbb{F}$-progressively measurable). Then, for any Euclidean space $\mathbb{H}$ and $t\in[0,T]$, we introduce the following spaces:
\begin{align*}
&C(t, T; \mathbb{H}) =\big\{\varphi:[t, T] \rightarrow \mathbb{H} \text{ }\big|\text{ } \varphi(\cdot) \text { is continuous }\big\}, \\
&L^{\infty}(t, T ; \mathbb{H}) =\big\{\varphi:[t, T] \rightarrow \mathbb{H} \text{ }\big|\text{ } \operatorname{esssup}_{s \in[t, T]}| \varphi(s)|<\infty\big\},\\
&L_{\mathcal{F}_{T}}^{2}(\Omega ; \mathbb{H}) =\big\{\xi: \Omega \rightarrow \mathbb{H} \text{ }\big|\text{ } \xi \text { is } \mathcal{F}_{T} \text {-measurable, } \mathbb{E}|\xi|^{2}<\infty\big\},\\
&\mathcal{S}_{\mathbb{F}}^{2}(t, T ; \mathbb{H}) =\big\{\varphi:[t, T] \times \Omega \rightarrow \mathbb{H} \text{ }\big|\text{ }\varphi(\cdot) \in \mathbb{F} \text {, } \mathbb{E}\big[\sup _{s \in[t, T]}|\varphi(s)|^{2}\big]<\infty\big\}, \\
&L_{\mathbb{F}}^{2}(t, T ; \mathbb{H}) =\big\{\varphi:[t, T] \times \Omega \rightarrow \mathbb{H} \text{ }\big|\text{ } \varphi(\cdot) \in \mathbb{F}\text{, } \mathbb{E} \int_{t}^{T}|\varphi(s)|^{2} ds<\infty\big\}, \\
&L_{\mathcal{P}}^{2}\left(t, T ; \mathbb{H}\right)=\big\{\varphi:[t, T] \times \Omega \rightarrow \mathbb{H} \text{ }\big|\text{ } \varphi(\cdot) \in \mathcal{P}\text{, } \mathbb{E} \int_{t}^{T}|\varphi(s)|^{2} ds<\infty\big\}.
\end{align*}

In this paper, we will study an SLQ optimal control problem with the following state equation and cost functional:
 \begin{equation}\label{state}
   \left\{
   \begin{array}{l}
   dX(s)=\left[A(s,\alpha(s))X(s)+B(s,\alpha(s))u(s)+b(s)\right]ds\\[2mm]
   \qquad\qquad+\left[C(s,\alpha(s))X(s)+D(s,\alpha(s))u(s)+\sigma(s)\right]dW(s)\\[2mm]
   \qquad\qquad+\sum_{j=1}^{L}\left[E_{j}(s,\alpha(s-))X(s-)+F_{j}(s,\alpha(s-))u(s)+\gamma_{j}(s)\right]d\widetilde{N}_{j}(s),\qquad s\in[t,T],\\[2mm]
   X(t)=x,\quad \alpha(t)=i,
   \end{array}
   \right.
 \end{equation}
 and
 \begin{equation}\label{cost}
\begin{aligned}
   J(t,x,i;u(\cdot))
   \triangleq \mathbb{E}\Big\{&\int_{t}^{T}\left[
    \left<
    \left(
    \begin{matrix}
    Q(s,\alpha(s)) & S(s,\alpha(s))^{\top}  \\
    S(s,\alpha(s)) & R(s,\alpha(s))
    \end{matrix}
    \right)
    \left(
    \begin{matrix}
    X(s) \\
    u(s)
    \end{matrix}
   \right),
    \left(
    \begin{matrix}
    X(s) \\
    u(s)
    \end{matrix}
    \right)
    \right>\right.\\
   &\left. +2\left<
    \left(
    \begin{matrix}
    q(s) \\
    \rho(s)
    \end{matrix}
   \right),
    \left(
    \begin{matrix}
    X(s) \\
    u(s)
    \end{matrix}
    \right)
   \right>
   \right]ds
   +\big<G(T,\alpha(T))X(T)+2g, X(T)\big>\Big\}.
  \end{aligned}
\end{equation}
In the above, $X(\cdot)$ is called the \emph{state process}  with initial value $(x,i)\in\mathbb{R}^{n}\times\mathcal{S}$, and $u(\cdot)$ is called the \emph{control process}, which belongs to the Hilbert space  $\mathcal{U}[t,T]\triangleq L_{\mathcal{P}}^{2}\left(t, T ; \mathbb{R}^{m}\right)$. Throughout this paper, we suppose  that the coefficients of the state equation \eqref{state} and cost functional \eqref{cost} satisfy the following assumptions:

 \textbf{(A1).} For any fixed $i,\text{ }j\in\mathcal{S}$,
 $$
 \begin{aligned}
 &A(\cdot,i), \text{ }C(\cdot,i), \text{ } E_{j}(\cdot,i)\in L^{\infty}(0,T;\mathbb{R}^{n\times n}),\quad B(\cdot,i),\text{ }D(\cdot,i),\text{ }F_{j}(\cdot,i)\in L^{\infty}( 0,T; \mathbb{R}^{n\times m}),\\
  &b(\cdot),\text{ }\sigma(\cdot)\in L_{\mathbb{F}}^{2}(0,T;\mathbb{R}^{n}),\quad
  \gamma_{j}(\cdot)\in L_{\mathcal{P}}^{2}(0,T;\mathbb{R}^{n}).
  \end{aligned}
  $$

 \textbf{(A2).} For any $i\in\mathcal{S}$,
\begin{align*}
&Q(\cdot,i)\in L^{\infty}(0,T;\mathbb{S}^{n}),\quad S(\cdot,i)\in L^{\infty}(0,T;\mathbb{R}^{m\times n }),\quad R(\cdot,i)\in L^{\infty}( 0,T; \mathbb{S}^{m}),\quad G(T,i)\in\mathbb{S}^{n},\\
&q(\cdot)\in L_{\mathbb{F}}^{2}(0,T;\mathbb{R}^{n}),\quad
\rho(\cdot)\in L_{\mathbb{F}}^{2}(0,T;\mathbb{R}^{m}),\quad
 g\in L_{\mathcal{F}_{T}}^{2}(\Omega,\mathbb{R}^{n})
\end{align*}
  Then, by Situ  \cite[Theorem 117]{situ2005theory}, for  given initial value $(t,x,i)\in[0,T]\times\mathbb{R}^{n}\times\mathcal{S}$ and control $u(\cdot)\in\mathcal{U}[t,T]$, the state equation \eqref{state} admits a unique solution $X(\cdot;t,x,i,u)\in\mathcal{S}_{\mathbb{F}}^{2}(t,T;\mathbb{R}^{n})$. Consequently, the cost functional is well-posed for any given $(t,x,i,u(\cdot))\in[0,T]\times\mathbb{R}^{n}\times\mathcal{S}\times \mathcal{U}[t,T]$. Based on the above notation, we summarize the SLQ optimal control problem for Markov regime-switching jump-diffusion systems as follows.

\noindent\textbf{Problem (M-SLQ):} For any given $(t,x,i)\in [0,T]\times\mathbb{R}^{n}\times \mathcal{S}$, find a $u^{*}(\cdot)\in\mathcal{U}[t,T]$ such that
\begin{equation}\label{value}
J\left(t,x,i;u^{*}(\cdot)\right)=\inf_{u(\cdot)\in \mathcal{U}[t,T]}J\left(t,x,i;u(\cdot)\right)\triangleq V(t,x,i).
\end{equation}
The function $V(\cdot,\cdot,\cdot)$ is called the value function of Problem (M-SLQ). If $b(\cdot)=\sigma(\cdot)=q(\cdot)=0$, $\rho(\cdot)=0$, $g=0$ and $\gamma_{j}(\cdot)=0$ for $j\in\mathcal{S}$, the corresponding state process, performance functional,  value function and problem are denoted by  $X^{0}(\cdot;t, x, i,u)$, $J^{0}(t, x, i;u)$, $V^{0}(\cdot,\cdot,\cdot)$, Problem (M-SLQ)$^0$,  respectively.

Linear-quadratic (LQ) control theory, as the most mature fundamental subfield within the system control theory, constitutes a vital cornerstone of modern control theory. Its concepts, methods, and results have exerted a profound influence on the development of other branches. Research on LQ control problems traces back to the seminal works of Kalman \cite{kalman1960contributions} and Letov \cite{Letov1961analytical}, with early studies primarily focused on deterministic control problems where state dynamics are governed by linear ordinary differential equations. In this context, the positive semi-definiteness of the quadratic control coefficient $R$ in the cost functional constitutes a necessary condition for the well-posedness of the LQ control problem. Notably, when $R$ is uniformly positive definite, the LQ control problem admits a complete solution via the Riccati equation under relatively weak conditions (see Yong and Zhou \cite{yong-zhou-2012-stochastic-control-theory}). In 1968, Wonham \cite{Wonham.W.M.1968} pioneered the study of SLQ control problems, which subsequently gained significant research attention due to their broad applications in engineering, finance, insurance, and economics. Influenced by deterministic LQ research, early investigations of SLQ problems predominantly assumed the weighting coefficients in the performance index satisfied \emph{standard conditions}. Under these assumptions, optimal controls for diffusion-based SLQ problems could be expressed as linear state feedback forms through solutions to Riccati equations, with the solvability of these equations being verifiable under standard conditions (see the monograph Yong and Zhou \cite{yong-zhou-2012-stochastic-control-theory}). Chen et al. \cite{Chen.S.P.1998_ILQ} discovered in their study of diffusion-based SLQ control that the quadratic control coefficient $R$ could even be negative definite when the control variable enters the diffusion term. In such cases, the failure of standard conditions renders the solvability of the associated Riccati equation a significant theoretical challenge. The authors established the solution under specific circumstances and subsequently constructed the optimal feedback representation. Following this breakthrough, numerous researchers have advanced LQ control theory by examining SLQ problems under diverse modeling frameworks. Below, we outline only those contributions most relevant to the present study.

Sun et al.  \cite{Sun-Li-Yong-2016} systematically investigated the indefinite SLQ control problem for diffusion systems over a finite time horizon, providing an in-depth analysis of the open-loop solvability and closed-loop solvability of the control problem. Their research demonstrates that the open-loop solvability of the control problem is equivalent to the cost functional satisfying a convexity condition and the solvability of the optimality system (a coupled system of forward-backward stochastic differential equations, FBSDEs). Conversely, the closed-loop solvability is equivalent to the existence of a regular solution to the corresponding differential Riccati equation. Notably, under the condition that the cost functional is uniformly convex (a requirement strictly weaker than the standard condition), the authors constructed a strongly regular solution to the differential Riccati equation using an iterative method based on solutions to Lyapunov equations. Inspired by these developments, Zhang et al. \cite{Zhang-Li-Xiong-2021} recently studied the indefinite SLQ control problem within regime-switching systems, providing a detailed discussion of open-loop and closed-loop solvability. Recently, Alia and Alia \cite{alia_time-inconsistent_2024} incorporated both Markovian jumps and Poisson jumps into the state equation while studying SLQ control for regime-switching systems. However, the author did not introduce a control variable in the associated Markovian jumps term (that is to say, the corresponding state coefficients in \cite{alia_time-inconsistent_2024} satisfy $F_{j}(\cdot,i)\equiv 0$ for all $i,\,j\in\mathcal{S}$).
%A significant observation is that constructing a closed-loop optimal strategy now requires solving a system of coupled differential Riccati equations (CDREs), whose solvability presents greater complexity than that of a single Riccati equation.
Additional research on SLQ control problems within regime-switching systems can be found in works such as Zhang and Yin \cite{Zhang.Q.1999_LQG}, Li and Zhou \cite{li_indefinite_2002}, and Wen et al. \cite{wen_weak_2021,Wen_2023}.

%{\color{red}
However, in the existing literature on SLQ problems based on Markov regime-switching systems, the state equation is commonly simplified to a diffusion system modulated by a Markov chain:
\begin{equation}\label{MRSDS}
\left\{
\begin{aligned}
dX(s)&=\left[A(s,\alpha(s))X(s)+B(s,\alpha(s))u(s)+b(s)\right]ds\\
&\quad+\left[C(s,\alpha(s))X(s)+D(s,\alpha(s))u(s)+\sigma(s)\right]dW(s),\\
 X(t)&=x,\quad \alpha(t)=i.
\end{aligned}
\right.
\end{equation}
Although the systems \eqref{MRSDS} can effectively describe the dynamic behavior of the state process switching between different modes (represented by different states of the Markov chain), it cannot capture abrupt jumps such as \emph{losses} or \emph{gains} that occur during mode switching since the system \eqref{MRSDS} remains a continuous stochastic system. In this paper, we introduce the following Markovian jumps into the state equation
\begin{equation}\label{Markovian-jumps}
\sum_{j=1}^{L}\left[E_{j}(s,\alpha(s-))X(s-)+F_{j}(s,\alpha(s-))u(s)+\gamma_{j}(s)\right]d\widetilde{N}_{j}(s).
\end{equation}
Hence, we can characterize both \emph{mode switching} and \emph{switching losses/gains} features of the control system. Obviously, when the jump coefficients in the Markov regime-switching jump-diffusion system satisfy the following condition:
$$
E_{j}(\cdot,i)=0,\quad F_{j}(\cdot,i)=0,\quad \gamma_{j}(\cdot)=0,\quad\forall i,\, j\in\mathcal{S},
$$
the Markov regime-switching jump-diffusion system reduces to the classical diffusion system modulated by a Markov chain. Therefore, the Markov regime-switching jump-diffusion system \eqref{state} can be regarded as a natural extension of the diffusion system \eqref{MRSDS} modulated by a Markov chain, and it is suitable for modeling more general stochastic optimal control problems.

 It is worthy to note that introducing Markovian jumps into the state equation has significant practical value. In application contexts, for stochastic optimal control problems with regime switching, the state process may typically experience jump-like mutations at the moment of a regime switch. Such jump sizes can be interpreted as the losses or gains of the state process incurred by the regime transition. Consider financial market modeling as an example. While the classical diffusion system modulated by a Markov chain can effectively capture the evolution of asset prices under different market trends (see, e.g., Zhou and Yin \cite{Zhou-2003-MV}, Shen \cite{Shen-2024-optimal}, Zhang et al. \cite{Zhang-2010-portfolio}), their inherent continuity as stochastic systems renders them incapable of modeling the losses or gains incurred during trend transitions. Extensive empirical evidence indicates that major events, including policy announcements, credit defaults, black swan events, and market crashes, usually coincide with market trend transitions and lead to instantaneous jumps in asset prices and portfolio values. Savku and Weber \cite{savku2022stochastic} and Shi and Xu \cite{shi2025optimal} observed the jump phenomenon of financial assets during market regime switching when discussing portfolio problems in financial markets. Accordingly, they incorporated Markovian jumps into the price equation for risky assets, thereby characterizing the losses and gains in risky asset prices accompanying market trend shifts.
In addition, Song and Wu \cite{song2022general} incorporated Markovian jumps into the state equation when studying the maximum principle for stochastic optimal control problems governed by Markov regime-switching systems, and established the maximum principle for general stochastic optimal control problems. Further studies on Markov regime-switching jump-diffusion systems in areas including queueing theory and financial engineering can be found in Azam et al. \cite{A} and Shu et al. \cite{B}. It should be noted that although a Poisson jump-diffusion system modulated by a Markov chain can describe the jump behavior of the state process, its jump term is driven by an exogenous Poisson process and thus cannot capture the discontinuous impact on the state process caused by the intrinsic regime switching of the control system.

Introducing the Markovian jumps into the state equation also poses new challenges for solving Problem (M-SLQ). In the study of SLQ control problems with state equation \eqref{MRSDS}, since the Markovian jumps \eqref{Markovian-jumps} are not introduced into the state equation, although the corresponding adjoint equation itself is a backward stochastic differential equation (BSDE) driven by both the Markovian jumps and Brownian motion, the variable $\Gamma_{j}(\cdot)$ characterizing the jump size does not appear in the generator of the corresponding adjoint equation and the corresponding stationary condition (see the second and third equations in \eqref{optimality-system-zhang}). By introducing the Markovian jumps \eqref{Markovian-jumps} into the state equation, this paper extends the corresponding adjoint equation to a more general class of BSDE, whose generator explicitly includes the variable corresponding to the jump size. This structural change significantly complicates the decoupling process of the optimality system for the SLQ control problem, which typically requires the regularity solution of a more complex Riccati equation. As we can see from the Remark \ref{rmk-2}, compared with the CDREs \eqref{CDREs-2} derived from the classical SLQ control problem with state equation \eqref{MRSDS}, the CDREs \eqref{CDREs} derived in this paper is highly coupled, and its cannot be directly addressed by the classical dimension-augmentation approach. In studying the SLQ control problem with Markovian jumps, Alia and Alia \cite{alia_time-inconsistent_2024}
%incorporated both Markovian jumps and Poisson jumps into the state equation while studying SLQ control for regime-switching systems. However, the author did not introduce control variable in the associated Markovian jumps term (that is to say, the corresponding state coefficients in \cite{alia_time-inconsistent_2024} satisfy $F_{j}(\cdot,i)\equiv 0$ for all $i\,j\in\mathcal{S}$), and they
explicitly acknowledged that the solvability of the corresponding stochastic differential Riccati equation becomes highly challenging when the control variable affects the Markov chain jump component (see Remark 2.10 in  \cite{alia_time-inconsistent_2024}). In studying the maximum principle for general stochastic optimal control problems, Zhang et al. \cite{Zhang-Elliott-Siu-2012} also introduces the Markovian jumps into the state equation. However, when discussing the mean-variance portfolio problem (a special case of the SLQ control problem) in financial markets, they only consider Poisson jumps modulated by a Markov chain in the asset price, without further introducing Markovian jumps.  All these facts reflect that adding controlled Markovian jumps \eqref{Markovian-jumps} to the state equation will bring certain difficulties and challenges to the solvability of Problem (M-SLQ).

In light of these considerations, this paper investigates an SLQ control problem within the framework of regime-switching jump-diffusion systems. Specifically, we postulate that both the parameters of the state equation and the state process itself undergo abrupt jumps concomitant with the Markovian regime transitions (see \eqref{state}). Unlike classical regime-switching diffusion systems \eqref{MRSDS}, these jumps in the state process during regime transitions can be interpreted as direct gains or losses induced by the regime switch to the system. The main contributions of this paper can be concluded as follows:
\begin{enumerate}
\item We model the SLQ control problem, incorporating both mode switching and switching losses/gains, using a regime-switching jump-diffusion system. Under the uniform convexity condition of the cost functional, the unique open-loop solvability of the control problem is established. Our results extend the study \cite{Zhang-Li-Xiong-2021} of the SLQ control problem for the regime-switching jump-diffusion system.
\item We establish the equivalence between the uniform convexity condition \eqref{uniformly-convex-condition} of the cost functional and the unique strongly regular solvability of corresponding CDREs \eqref{CDREs}. Based on this result, we provide a closed-loop representation for open-loop optimal control and an explicit expression for the value function.
\item We apply our findings to a mean-variance portfolio selection problem. Under the assumption that the risky asset prices follow a regime-switching jump-diffusion model, we successfully derive the efficient frontier for the portfolio problem based on the derived results. From this perspective, our findings develop the results in Zhou and Yin  \cite{Zhou-2003-MV}, which studied the mean-variance portfolio selection problem under the assumption that the risky asset prices follow a regime-switching diffusion model.
\end{enumerate}

The rest of the paper is organized as follows. %Section \ref{section-2} introduces the basic notations and formulates the problem studied in this paper.
In Section \ref{section-3}, we characterize the cost functional from the Hilbert space point of view. Section \ref{section-4} studies the open-loop solvability in terms of the optimality system and the convexity condition for the cost functional. Section \ref{section-5} aims to investigate the closed-loop representation for open-loop optimal control based on the solution to a system of CDREs, whose solvability is studied in  Section \ref{section-6}. Finally, Section \ref{section-7} provides an application to the mean-variance portfolio selection problem, shedding light on how to employ the obtained results.
\section{Representation of the cost functional}\label{section-3}
We begin this section by introducing some useful notations besides those introduced in the previous section. We denote $\mathbb{R}^{n \times m}$ as the Euclidean space of all $n \times m  $ matrices and $\mathbb{R}^{n}$ as the $\mathbb{R}^{n \times 1}$ for simplicity. In addition, the set of all $n\times n$ symmetric matrices is denoted by $\mathbb{S}^{n}$. Specially, the sets of all $ n\times n$ semi-positive definite matrices and positive definite matrices are denoted by $\overline{\mathbb{S}_{+}^n}$ and $\mathbb{S}_{+}^n$, respectively. We suppose that the Hilbert space $\mathbb{R}^{n\times m}$ is equipped with the Frobenius inner product $<M,N>\triangleq tr(M^{\top}N)$ for any $M,\, N\in\mathbb{R}^{n \times m}$, where $M^{\top}$ is the transpose of $M$ and $tr(M^{\top}N)$ is the trace of $M^{\top}N$. The norm introduced by the Frobenius inner product is denoted by $|\cdot|$. The identity matrix of size $n$ is denoted by $I_{n}$, which is often written as $I$ when no confusion occurs. For any $M, N \in \mathbb{S}^n$, we write $M \geqslant N$ (respectively, $M>N$) if $M-N$ is semi-positive definite (respectively, positive definite). Specially, for an $\mathbb{S}^n$-valued measurable function $F(\cdot)$ on $[0, T]$, we write
$$\left\{
\begin{array}{lll}
F(\cdot) \geqslant 0 & \text { if } \quad F(s) \geqslant 0, & \text { a.e. } s \in[0, T], \\
F(\cdot)>0 & \text { if } \quad F(s)>0, & \text { a.e. } s \in[0, T], \\
F(\cdot) \gg 0 & \text { if } \quad F(s) \geqslant \delta I_n, & \text { a.e. } s \in[0, T], \quad \text { for some } \delta>0 .
\end{array}
\right.$$
Moreover, we use $F(\cdot)\leq 0$, $F(\cdot)<0$ and $F(\cdot)\ll 0$ to indicated that $-F(\cdot)\geq 0$, $-F(\cdot)>0$ and $-F(\cdot)\gg 0$, respectively. For any given Banach space $\mathbb{B}$, we denote
$$\mathcal{D}\left(\mathbb{B}\right)=\left\{\mathbf{\Lambda}(\cdot)=\left(\Lambda(\cdot,1),\cdots,\Lambda(\cdot,L)\right) \mid \Lambda(\cdot,i) \in \mathbb{B}\text{, } \forall i\in \mathcal{S}\right\}.$$

Now, we  consider the following FBSDEs:
\begin{equation}\label{FBSDE}
      \left\{
      \begin{array}{l}
      dX(s)=\left[A(s,\alpha(s))X(s)+B(s,\alpha(s))u(s)+b(s)\right]ds\\[2mm]
      \qquad\qquad +\left[C(s,\alpha(s))X(s)+D(s,\alpha(s))u(s)+\sigma(s)\right]dW(s)\\[2mm]
      \qquad\qquad+\sum_{j=1}^{L}\left[E_{j}(s,\alpha(s-))X(s-)+F_{j}(s,\alpha(s-))u(s)
      +\gamma_{j}(s)\right]d\widetilde{N}_{j}(s),\\[2mm]
      dY(s)=-\big[A(s,\alpha(s))^{\top}Y(s)+C(s,\alpha(s))^{\top}Z(s)+\sum_{j=1}^{L}\lambda_{j}(s) E_{j}(s,\alpha(s))^{\top}\Gamma_{j}(s)\\[2mm]
      \qquad\qquad+ Q(s,\alpha(s))X(s)+S(s,\alpha(s))^{\top}u(s)+q(s)\big]ds+Z(s)dW(s)+\sum_{j=1}^{L}\Gamma_{j}(s) d\widetilde{N}_{j}(s),\\[2mm]
      X(t)=x,\quad \alpha(t)=i,\quad Y(T)=G(T,\alpha(T)) X(T)+g,\quad s\in[0,T].
      \end{array}
      \right.
    \end{equation}
To emphasize the relationship between the solution of the FBSDEs \eqref{FBSDE} and the specified initial state $(t,x,i)\in[0,T]\times\mathbb{R}^{n}\times \mathcal{S}$ along with the control $u(\cdot)\in\mathcal{U}[t,T]$, %To emphasize the connection between the solution of above FBSDEs and the given condition $(t,x,i,u(\cdot))\in[0,T]\times\mathbb{R}^{n}\times \mathcal{S}\times\mathcal{U}[t,T]$,
we denote the solution to \eqref{FBSDE} by $$\left(X(\cdot;t,x,i,u),Y(\cdot;t,x,i,u),Z(\cdot;t,x,i,u),\mathbf{\Gamma}(\cdot;t,x,i,u)\right),$$
with
$$\mathbf{\Gamma}(\cdot;t,x,i,u)\triangleq\left[\Gamma_{1}(\cdot;t,x,i,u),\Gamma_{2}(\cdot;t,x,i,u),
\cdots,\Gamma_{L}(\cdot;t,x,i,u)\right].$$
Additionally, if $b(\cdot)=\sigma(\cdot)=q(\cdot)=0$, $\rho(\cdot)=0$, $g=0$ and $\gamma_{j}(\cdot)=0$ for $j\in\mathcal{S}$, we denote the associated solution to FBSDEs \eqref{FBSDE} as
$$
\left(X^{0}(\cdot;t,x,i,u),Y^{0}(\cdot;t,x,i,u),Z^{0}(\cdot;t,x,i,u),\mathbf{\Gamma}^{0}(\cdot;t,x,i,u)\right).
$$

Based on the above notations, we can provide a representation of the cost functional \eqref{cost} from the Hilbert space point of view.
\begin{proposition}\label{prop-cost-representation}
Let assumptions (A1)-(A2) hold. Then, for any $(t,x,i,u(\cdot))\in[0,T]\times\mathbb{R}^{n}\times \mathcal{S}\times\mathcal{U}[t,T]$, the cost functional \eqref{cost} can be represented as
\begin{equation}\label{cost-representation}
\begin{aligned}
&J^{0}(t,x,i;u(\cdot))=\big<M_{2}(t,i)u,u\big>+2\big<M_{1}(t,i)x,u\big>+\big<M_{0}(t,i)x,x\big>,\\
&J(t,x,i;u(\cdot))=\big<M_{2}(t,i)u,u\big>+2\big<M_{1}(t,i)x,u\big>+\big<M_{0}(t,i)x,x\big>
+2\big<\nu_{t}^{i},u\big>+2\big<y_{t}^{i},x\big>+c_{t}^{i},
\end{aligned}
\end{equation}
where
\begin{align*}
\big[M_{2}(t,i)u\big](s)&=B(s,\alpha(s))^{\top}Y^{0}(s;t,0,i,u)+D(s,\alpha(s))^{\top}Z^{0}(s;t,0,i,u)\\
&\quad+\sum_{j=1}^{L}\lambda_{j}(s) F_{j}(s,\alpha(s))^{\top}\Gamma_{j}^{0}(s;t,0,i,u)+S(s,\alpha(s))X^{0}(s;t,0,i,u)+R(s,\alpha(s))u(s),\\
\big[M_{1}(t,i)x\big](s)&=B(s,\alpha(s))^{\top}Y^{0}(s;t,x,i,0)+D(s,\alpha(s))^{\top}Z^{0}(s;t,x,i,0)\\
 &\quad+\sum_{j=1}^{L}\lambda_{j}(s) F_{j}(s,\alpha(s))^{\top}\Gamma_{j}^{0}(s;t,x,i,0)+S(s,\alpha(s))X^{0}(s;t,x,i,0),\\
M_{0}(t,i)x&=\mathbb{E}\left[Y^{0}(t;t,x,i,0)\right],\\
\nu_{t}^{i}(s)&=B(s,\alpha(s))^{\top}Y(s;t,0,i,0)+D(s,\alpha(s))^{\top}Z(s;t,0,i,0)\\
&\quad+\sum_{j=1}^{L}\lambda_{j}(s) F_{j}(s,\alpha(s))^{\top}\Gamma_{j}(s;t,0,i,0)+S(s,\alpha(s))X(s;t,0,i,0)+\rho(s),\\
y_{t}^{i}&=\mathbb{E}\left[Y(s;t,0,i,0)\right],\\
c_{t}^{i}&=\mathbb{E}\Big\{\int_{t}^{T}\big<Q(s,\alpha(s))X(s;t,0,i,0)+2q(s),X(s;t,0,i,0)\big>ds\\
&\qquad\quad +\big<G(T,\alpha(T))X(T;t,0,i,0)+2g,X(T;t,0,i,0)\big>\Big\}.
\end{align*}
\end{proposition}
\begin{proof}
  Observe that
   $$X^{0}(\cdot;t,x,i,u)=X^{0}(\cdot;t,0,i,u)+X^{0}(\cdot;t,x,i,0).$$
   Therefore, $J^{0}(t,x,i;u(\cdot))$ can be decomposed into
   \begin{equation*}
   J^{0}(t,x,i;u(\cdot))=I_{1}+2I_{2}+I_{3},
   \end{equation*}
   where
    \begin{align*}
      &I_{1}=\mathbb{E}\Big\{\int_{t}^{T}\Big[\big< Q(s,\alpha(s))X^{0}(s;t,0,i,u), X^{0}(s;t,0,i,u)\big>
      +2\big<S(s,\alpha(s))X^{0}(s;t,0,i,u),u(s)\big>\\
     & \qquad+\big< R(s,\alpha(s))u(s), u(s)\big>\Big] ds
      +\big< G(T,\alpha(T)) X^{0}(T;t,0,i,u), X^{0}(T;t,0,i,u)\big>\Big\},\\
      &I_{2}=\mathbb{E}\Big\{\int_{t}^{T}\Big[\big< Q(s,\alpha(s))X^{0}(s;t,x,i,0), X^{0}(s;t,0,i,u)\big>
      +\big<S(s,\alpha(s))X^{0}(s;t,x,i,0),u(s)\big>\Big] ds\\
      &\qquad+\big< G(T,\alpha(T)) X^{0}(T;t,x,i,0), X^{0}(T;t,0,i,u)\big>\Big\},\\
      &I_{3}=\mathbb{E}\Big\{\int_{t}^{T}\big< Q(s,\alpha(s))X^{0}(s;t,x,i,0), X^{0}(s;t,x,i,0)\big>ds\\
      &\qquad +\big< G(T,\alpha(T)) X^{0}(T;t,x,i,0), X^{0}(T;t,x,i,0)\big>\Big\}.
    \end{align*}
    Consequently, applying the It\^{o}'s rule to
    $\big<Y^{0}(s;t,0,i,u), X^{0}(s;t,0,i,u)\big>$, $\big<Y^{0}(s;t,x,i,0) ,X^{0}(s;t,0,i,u)\big>$,  and  $\big<Y^{0}(s;t,x,i,0) ,X^{0}(s;t,x,i,0)\big>$,
    respectively, we obtain
    \begin{align*}
      &I_{1}=\mathbb{E}\int_{t}^{T}
      \big< \left[M_{2}(t,i)u\right](s), u(s)\big> ds
      =\big< M_{2}(t,i)u, u\big>,\\
      &I_{2}=\mathbb{E}\int_{t}^{T}
      \big< \left[M_{1}(t,i)x\right](s), u(s)\big> ds
      =\big< M_{1}(t,i)x, u \big>,\\
      &I_{3}=\big<M_{0}(t,i)x, x\big>.
    \end{align*}

  Similarly, we can decompose $J(t,x,i;u(\cdot))$ as %into
  \begin{equation*}
  J(t,x,i;u(\cdot))=J^{0}(t,x,i;u(\cdot))+2I_{4}+2I_{5}+I_{6},
  \end{equation*}
  where
   \begin{align*}
     I_{4}&=\mathbb{E}\Big\{\int_{t}^{T}\Big[\big< Q(s,\alpha(s))X(s;t,0,i,0)+q(s), X^{0}(s;t,0,i,u)\big>\\
     &\quad +\big<S(s,\alpha(s))X(s;t,0,i,0)+\rho(s),u(s)\big>\Big] ds\\
     &\quad+\big< G(T,\alpha(T)) X(T;t,0,i,0)+g, X^{0}(T;t,0,i,u)\big>\Big\}=\big<\nu_{t}^{i},u\big>,\\
     I_{5}&=\mathbb{E}\Big\{\int_{t}^{T}\big< Q(s,\alpha(s))X(s;t,0,i,0)+q(s), X^{0}(s;t,x,i,0)\big>ds\\
      &\quad+\big< G(T,\alpha(T)) X(T;t,0,i,0)+g, X^{0}(T;t,x,i,0)\big>\Big\}
      =\big<y_{t}^{i},x\big>,\\
     I_{6}&=\mathbb{E}\Big\{\int_{t}^{T}\big< Q(s,\alpha(s))X(s;t,0,i,0)+2q(s), X(s;t,0,i,0)\big>ds\\
      &\quad+\big< G(T,\alpha(T)) X(T;t,0,i,0)+2g, X(T;t,0,i,0)\big>\Big\}=c_{t}^{i}.
   \end{align*}
 This completes the proof.
\end{proof}

\begin{remark}\label{rmk-bounded}\rm
Based on the estimation for the solution to FBSDE \eqref{FBSDE} (see, for example,  \cite{cohen2015stochastic,situ2005theory}), we point out that $M_{2}(t,i)$, $M_{1}(t,i)$ and $M_{0}(t,i)$ are bounded linear operators of different Hilbert space. More specifically, we have
\begin{equation*}
  M_{2}(t,i):\text{ } \mathcal{U}[t,T] \mapsto \mathcal{U}[t,T],\qquad
  M_{1}(t,i): \text{ }\mathbb{R}^{n} \mapsto \mathcal{U}[t,T],\qquad
  M_{0}(t,i): \text{ }\mathbb{R}^{n} \mapsto \mathbb{R}^{n}.
\end{equation*}
\end{remark}
Based on the  Proposition \ref{prop-cost-representation}, we can directly obtain the following corollary.
\begin{corollary}\label{coro-cost-representation}
Let assumptions (A1)-(A2) hold. Then for any $(t,x,i)\in[0,T]\times\mathbb{R}^{n}\times \mathcal{S}$, $u(\cdot), \upsilon(\cdot)\in \mathcal{U}[t,T]$ and $\varepsilon\in\mathbb{R}$, the following holds:
\begin{equation}\label{epsilon-cost}
 J(t,x,i;u(\cdot)+\varepsilon\upsilon(\cdot))=J(t,x,i,u(\cdot))+\varepsilon^2 J^{0}(t,0,i,\upsilon(\cdot))
 +2\varepsilon\mathbb{E}\int_{t}^{T}\big<\bar{M}[t,i,x,u](s),\upsilon(s)\big>ds,
\end{equation}
where
\begin{equation}\label{bar-M}
\hspace{-0.4cm}\begin{array}{l}
  \bar{M}[t,i,x,u](s)=B(s,\alpha(s))^{\top}Y(s;t,x,i,u)+D(s,\alpha(s))^{\top}Z(s;t,x,i,u)+R(s,\alpha(s))u(s)+\rho(s)\\[2mm]
\qquad\qquad\qquad\hspace{2mm}+\sum_{j=1}^{L}\lambda_{j}(s) F_{j}(s,\alpha(s))^{\top}\Gamma_{j}(s;t,x,i,u)
+S(s,\alpha(s))X(s;t,x,i,u).
\end{array}
\end{equation}
Consequently, the map $u(\cdot)\mapsto J(t,x,i;u(\cdot))$ is Fr\'echet differentiable and the Fr\'echet derivative is given by
\begin{equation}\label{derivative-cost}
\mathcal{D}J(t,x,i;u(\cdot))(s)=2\bar{M}[t,i,x,u](s).
\end{equation}
\end{corollary}
\begin{proof}
It follows from equation \eqref{cost-representation} that
\begin{equation}\label{eq-1}
  \begin{aligned}
 &\quad  J(t,x,i;u(\cdot)+\varepsilon\upsilon(\cdot))-J(t,x,i;u(\cdot))\\
  &=\varepsilon^{2}\big<M_{2}(t,i)\upsilon,\upsilon\big>
  +2\varepsilon\big[\big<M_{2}(t,i)u,\upsilon\big>+\big<M_{1}(t,i)x,\upsilon\big>+\big<\nu_{t}^{i},\upsilon\big>\big]\\
  &=\varepsilon^{2} J^{0}(t,0,i;\upsilon(\cdot))+2\varepsilon\big<M_{2}(t,i)u+M_{1}(t,i)x+\nu_{t}^{i},\upsilon\big>.
  \end{aligned}
  \end{equation}
  Based on the linearity of FBSDEs \eqref{FBSDE}, we can further obtain
  $$
  \begin{array}{l}
  \big[M_{2}(t,i)u+M_{1}(t,i)x+\nu_{t}^{i}\big](s)=B(s,\alpha(s))^{\top}Y(s;t,x,i,u)+D(s,\alpha(s))^{\top}Z(s;t,x,i,u)\\[2mm]
\qquad+\sum_{j=1}^{L}\lambda_{j}(s) F_{j}(s,\alpha(s))^{\top}\Gamma_{j}(s;t,x,i,u)
+S(s,\alpha(s))X(s;t,x,i,u)+R(s,\alpha(s))u(s)+\rho(s).
\end{array}
  $$
  Hence, the desired result follows by substituting the above equation into \eqref{eq-1}.
  %This completes the proof.
\end{proof}

\section{Open-loop solvability}\label{section-4}
In this section, we shall study the open-loop solvability of Problem (M-SLQ). The following result establishes the equivalence between open-loop solvability and the existence of an adapted solution to FBSDEs with constraints under the convexity condition on the cost functional.

\begin{theorem}\label{thm-open-loop-solvability}
Let assumptions (A1)-(A2) hold. An element $u(\cdot)\in \mathcal{U}[t,T]$ is an open-loop optimal control of Problem (M-SLQ) for the initial value $(t,x,i)\in [0,T]\times\mathbb{R}^{n}\times \mathcal{S}$ if and only if:
  \begin{description}
    \item[(i)] The following convexity condition holds:
    \begin{equation}\label{convex-condition}
      J^{0}(t,0,i;v(\cdot))\geq 0,\quad \forall v(\cdot)\in \mathcal{U}[t,T], \quad \forall i\in \mathcal{S};
    \end{equation}
    \item[(ii)] The adapted solution $\left(X(\cdot),Y(\cdot),Z(\cdot),\mathbf{\Gamma}(\cdot)\right)$ to the  FBSDEs \eqref{FBSDE} % with $u(\cdot)$ replaced by $u(\cdot)$
    satisfies the following stationary condition:
  \begin{equation}\label{stationary-condition}
  \begin{array}{l}
      \bar{M}[t,i,x,u](s)=B(s,\alpha(s))^{\top}Y(s)+ D(s,\alpha(s))^{\top}Z(s)
    +\sum_{j=1}^{L}\lambda_{j}(s)F_{j}(s,\alpha(s))^{\top}\Gamma_{j}(s)\\[2mm]
      \qquad+S(s,\alpha(s))X(s)+R(s,\alpha(s))u(s)+\rho(s)=0,\quad a.e.\text{ }s\in[t,T],\quad a.s. .
     \end{array}
  \end{equation}
  \end{description}
\end{theorem}
\begin{proof}
  By definition, $u(\cdot)\in \mathcal{U}[t,T]$ is an open-loop optimal control for Problem (M-SLQ) if and only if
  \begin{equation}\label{open-loop}
  J(t,x,i;u(\cdot)+\varepsilon \upsilon(\cdot))- J(t,x,i;u(\cdot))\geq 0,\quad \forall\varepsilon\in\mathbb{R},\quad \forall  \upsilon(\cdot)\in\mathcal{U}[t,T].
\end{equation}
It follows from Corollary \ref{coro-cost-representation} that
  %By Corollary \ref{coro-cost-representation}, we obtain
  $$
   J(t,x,i;u(\cdot)+\varepsilon\upsilon(\cdot))-J(t,x,i,u(\cdot))=\varepsilon^2 J^{0}(t,0,i,\upsilon(\cdot))
 +2\varepsilon\mathbb{E}\int_{t}^{T}\big<\bar{M}[t,i,x,u](s),\upsilon(s)\big>ds.
  $$
Consequently, it is clear that the condition \eqref{open-loop} holds if and only if the convexity condition \eqref{convex-condition} and the stationary condition \eqref{stationary-condition} hold.
\end{proof}

\begin{remark}\label{rmk-1}\rm
In the above, the stationary condition \eqref{stationary-condition} together with the FBSDEs \eqref{FBSDE} %(with $u(\cdot)$ replaced by $u^{*}(\cdot)$)
constitute the \emph{optimality system} for the Problem (M-SLQ). As we can see from Zhang et al. \cite{Zhang-Li-Xiong-2021}, the corresponding optimality system for the SLQ control problem with state \eqref{MRSDS} is given by
\begin{equation}\label{optimality-system-zhang}
      \left\{
      \begin{array}{l}
      dX(s)=\left[A(s,\alpha(s))X(s)+B(s,\alpha(s))u(s)+b(s)\right]ds\\[2mm]
      \qquad\qquad +\left[C(s,\alpha(s))X(s)+D(s,\alpha(s))u(s)+\sigma(s)\right]dW(s)\\[2mm]
      dY(s)=-\big[A(s,\alpha(s))^{\top}Y(s)+C(s,\alpha(s))^{\top}Z(s)+ Q(s,\alpha(s))X(s)\\[2mm]
      \qquad\qquad+S(s,\alpha(s))^{\top}u(s)+q(s)\big]ds+Z(s)dW(s)+\sum_{j=1}^{L}\Gamma_{j}(s) d\widetilde{N}_{j}(s),\\[2mm]
      B(s,\alpha(s))^{\top}Y(s)+ D(s,\alpha(s))^{\top}Z(s)+S(s,\alpha(s))X(s)+R(s,\alpha(s))u(s)+\rho(s)=0,\\[2mm]
      X(t)=x,\quad \alpha(t)=i,\quad Y(T)=G(T,\alpha(T)) X(T)+g,\quad s\in[0,T].
      \end{array}
      \right.
    \end{equation}
In the above, although the adjoint equation (the second equation in \eqref{optimality-system-zhang}) is jointly driven by the Markovian jumps and Brownian motion, the jump size variable $\Gamma_{j}(\cdot)$ does not appear in its generator. In contrast, the optimality system derived in this paper is significantly more complex. As will be seen in the subsequent analysis, decoupling such an optimality system requires the regularity solution of a more complex Riccati equation.
\end{remark}

Theorem \ref{thm-open-loop-solvability} only provides an equivalent characterization for open-loop optimal control by the solvability of the optimality system. However, the optimality system is a coupled FBSDEs due to the stationary condition \eqref{stationary-condition}, and its solvability is not self-evident. Therefore, a fundamental question naturally arises: Under what conditions is Problem (M-SLQ) guaranteed to be open-loop solvable, and what is the expression for its open-loop optimal control?

Both Sun et al. \cite{Sun-Li-Yong-2016} and Zhang et al. \cite{Zhang-Li-Xiong-2021} have proven that the SLQ control problems with deterministic coefficients are open-loop solvable under the uniform convexity condition of the cost functional. In what follows, we develop this result for Problem (M-SLQ). We first introduce the following  uniform convexity condition of the cost functional:
\begin{equation}\label{uniformly-convex-condition}
  J^{0}(0,0,i;u(\cdot))\geq \mu \mathbb{E}\int_{0}^{T}\big|u(s)\big|^{2}ds,\quad \forall u(\cdot)\in\mathcal{U}[0,T],\quad \forall i\in\mathcal{S},\quad \text{for some } \mu>0.
\end{equation}
We then obtain the following result.
\begin{lemma}\label{lem-0}
 Let assumptions (A1)-(A2) and \eqref{uniformly-convex-condition} hold. Then for any $t\in[0,T]$,  the following holds:
 \begin{equation}\label{uniformly-convex-condition-t}
  J^{0}(t,0,i;u(\cdot))\geq \mu \mathbb{E}\int_{t}^{T}\big|u(s)\big|^{2}ds,\quad \forall u(\cdot)\in\mathcal{U}[t,T],\quad \forall i\in\mathcal{S},\quad \text{for some } \mu>0.
\end{equation}
\end{lemma}
\begin{proof}
 For any $t\in[0,T]$ and $u(\cdot)\in \mathcal{U}[t,T]$, let
$$\upsilon(s)=\left\{
\begin{array}{ll}
0, & s\in[0,t)\\
u(s), & s\in[0,T].
\end{array}
\right.$$
Then $\upsilon(\cdot)\in\mathcal{U}[0,T]$ and $X^{0}(s;0,0,j,\upsilon)\equiv 0$ over the interval $[0,t]$ for all $j\in\mathcal{S}$.
Consequently, for any $i\in\mathcal{S}$, one has
\begin{align*}
J^{0}(t,0,i;u(\cdot))=\mathbb{E}\sum_{j=1}^{L}\left[J^{0}(0,0,j;\upsilon(\cdot))I_{\{\alpha(0)=j\}}\big| \alpha(t)=i\right]
\geq \mu \mathbb{E}\int_{0}^{T} |\upsilon(s)|^2ds
=\mu \mathbb{E}\int_{t}^{T} |u(s)|^2ds.
\end{align*}
This completes the proof.
\end{proof}

The following result points out that the Problem (M-SLQ) is open-loop solvable for any given initial value under the uniform convexity condition \eqref{uniformly-convex-condition}.

\begin{theorem}\label{thm-open-loop-solvability-uniform-convex}
  Let assumptions (A1)-(A2) and \eqref{uniformly-convex-condition} hold.  Then for any $(t,x,i)\in[0,T]\times \mathbb{R}^{n}\times\mathcal{S}$, Problem (M-SLQ) is uniquely open-loop solvable, and there exists a constant $\gamma\in\mathbb{R}$ such that
  \begin{equation}\label{value-0-property}
    V^{0}(t,x,i)\geq \gamma \big|x\big|^{2},\quad \forall (t,x,i)\in [0,T]\times \mathbb{R}^{n}\times\mathcal{S}.
  \end{equation}
\end{theorem}
\begin{proof}
By Corollary \ref{coro-cost-representation} and Lemma \ref{lem-0}, we have
\begin{equation}\label{J}
\begin{aligned}
 J(t,x,i;u(\cdot))&=J(t,x,i;0)+J^{0}(t,0,i;u(\cdot))+\mathbb{E}\int_{t}^{T}\big<\mathcal{D}J(t,x,i;0)(s),u(s)\big>ds\\
 &\geq J(t,x,i;0)+\frac{\mu}{2}\mathbb{E}\int_{t}^{T}\big|u(s)\big|^{2}ds
 -\frac{1}{2\mu}\mathbb{E}\int_{t}^{T}\big|\mathcal{D}J(t,x,i;0)(s)\big|^{2}ds.
\end{aligned}
\end{equation}
Thus, adopting the standard argument involving minimizing sequence and locally weak compactness of Hilbert space, we obtain that the Problem (M-SLQ) admits a unique open-loop optimal control for any given initial value $(t,x,i)\in[0,T]\times \mathbb{R}^{n}\times\mathcal{S}$.

Additionally, if $b(\cdot)=\sigma(\cdot)=q(\cdot)=0$, $\rho(\cdot)=0$, $g=0$, then \eqref{J} implies that
\begin{equation}\label{V}
V^{0}(t,x,i)\geq J^{0}(t,x,i;0)-\frac{1}{2\mu}\mathbb{E}\int_{t}^{T}\big|\mathcal{D}J^{0}(t,x,i;0)(s)\big|^{2}ds,
\end{equation}
where
$$
  \begin{array}{l}
\mathcal{D}J^{0}(t,x,i;u(\cdot))(s)=2\Big\{B(s,\alpha(s))^{\top}Y^{0}(s;t,x,i,0)+D(s,\alpha(s))^{\top}Z^{0}(s;t,x,i,0)\\[2mm]
\qquad\qquad+\sum_{j=1}^{L}\lambda_{j}(s) F_{j}(s,\alpha(s))^{\top}\Gamma_{j}^{0}(s;t,x,i,0)
+S(s,\alpha(s))X^{0}(s;t,x,i,0)\Big\}.
\end{array}
$$
The functions on the right-hand side of \eqref{V} are quadratic in $x$ and continuous in $t$. Hence, the desired result \eqref{value-0-property} follows immediately.
\end{proof}

The above theorem illustrates that the equation \eqref{uniformly-convex-condition} serves as a sufficient condition for the open-loop solvability of Problem (M-SLQ). However, from an application perspective, condition \eqref{uniformly-convex-condition} remains challenging to verify. As demonstrated in  \cite{Zhang-Li-Xiong-2021} for the SLQ control problem of regime-switching diffusion systems, the following standard condition guarantees the validity of condition \eqref{uniformly-convex-condition}:
\begin{equation}\label{standard-condition}
  G(T,i)\geq 0,\quad R(\cdot,i)\gg 0,\quad Q(\cdot,i)-S(\cdot,i)^{\top}R(\cdot,i)S(\cdot,i)\geq 0,\quad i\in\mathcal{S}.
\end{equation}

Next, we also develop this result for the Problem (M-SLQ). To this end, we first provide the following lemma.

\begin{lemma}\label{lem-3}
 Suppose that assumption (A1) holds and $u(\cdot)\in \mathcal{U}[t,T]$. Let $X^{u}(\cdot)$ be the solution to
 \begin{equation}\label{lemma-3-1}
   \left\{
   \begin{array}{l}
   dX^{u}(s)=\big[A(s,\alpha(s))X^{u}(s)+B(s,\alpha(s))u(s)\big]ds
   +\big[C(s,\alpha(s))X^{u}(s)+D(s,\alpha(s))u(s)\big]dW(s)\\[2mm]
   \qquad\qquad+\sum_{j=1}^{L}\big[E_{j}(s,\alpha(s-))X^{u}(s)
   +F_{j}(s,\alpha(s))u(s)\big]d\widetilde{N}_{j}(s), \quad s\in[t,T],\\[2mm]
   X^{u}(t)=0,\quad \alpha(t)=i.
   \end{array}
   \right.
 \end{equation}
 Then, for any  $\mathbf{\Theta}(\cdot)\in \mathcal{D}\left(L^{\infty}(t,T;\mathbb{R}^{m\times n})\right)$, there exist a constant $\delta>0$ such that
 \begin{equation}\label{lem-3-2}
   \mathbb{E}\int_{t}^{T}\big|u(s)-\Theta(s,\alpha(s))X^{u}(s)\big|^{2}ds\geq \delta \mathbb{E}\int_{t}^{T}\big|u(s)\big|^{2}ds.
 \end{equation}
\end{lemma}

\begin{proof}
  For given $\mathbf{\Theta}(\cdot)\in \mathcal{D}\left(L^{\infty}(t,T;\mathbb{R}^{m\times n})\right)$, one can define a bounded linear operator $\mathfrak{L}:$ $\mathcal{U}[t,T]\rightarrow \mathcal{U}[t,T]$ by
  $$\mathfrak{L} u(s)=u(s)-\Theta(s,\alpha(s))X^{u}(s),\quad s\in[t,T].$$
  Then $\mathfrak{L}$ is bijective, and its inverse $\mathfrak{L}^{-1}$ is given by
  $$
  \mathfrak{L}^{-1} u(s)=u(s)+\Theta(s,\alpha(s))\widetilde {X}^{u}(s), \quad s\in[t,T],
  $$
  where $\widetilde {X}^{u}(\cdot)$ is the solution to the following SDE:
  \begin{equation}\label{lemma-3-2}
   \left\{
   \begin{array}{l}
   d\widetilde {X}^{u}(s)=\big\{\big[A(s,\alpha(s))+B(s,\alpha(s))\Theta(s,\alpha(s))\big]\widetilde {X}^{u}(s)+B(s,\alpha(s))u(s)\big\}ds\\[2mm]
   \qquad+\big\{\big[C(s,\alpha(s))+D(s,\alpha(s))\Theta(s,\alpha(s))\big]\widetilde {X}^{u}(s)+D(s,\alpha(s))u(s)\big\}dW(s)\\[2mm]
   \qquad+\sum_{j=1}^{L}\big\{\big[E_{j}(s,\alpha(s-))+F_{j}(s,\alpha(s-))\Theta(s,\alpha(s-))\big]\widetilde {X}^{u}(s-)+F_{j}(s,\alpha(s-))u(s)\big\}d\widetilde{N}_{j}(s), \\[2mm]
   \widetilde {X}^{u}(t)=0,\quad \alpha(t)=i,\quad s\in[t,T].
   \end{array}
   \right.
 \end{equation}
 By the bounded inverse theorem, $\mathfrak{L}^{-1}$ is a bounded operator with  norm $||\mathfrak{L}^{-1}||>0$. Hence,
 \begin{align*}
  \mathbb{E}\int_{t}^{T}\big|u(s)\big|^{2}ds
  &=\mathbb{E}\int_{t}^{T}\big|\mathfrak{L}^{-1}(\mathfrak{L}u)(s)\big|^{2}ds
  \leq ||\mathfrak{L}^{-1}||\mathbb{E}\int_{t}^{T}\big|\mathfrak{L}u(s)\big|^{2}ds\\
  &=||\mathfrak{L}^{-1}||\mathbb{E}\int_{t}^{T}\big|u(s)-\Theta(s,\alpha(s))X^{u}(s)\big|^{2}ds,\quad \forall
  u(\cdot)\in\mathcal{U}[t,T].
 \end{align*}
 Consequently, the desired result follows by selecting $\delta=||\mathfrak{L}^{-1}||^{-1}>0$.
\end{proof}

\begin{proposition}
  Let assumptions (A1)-(A2) hold. Then the standard condition \eqref{standard-condition}
  implies the uniform convexity condition \eqref{uniformly-convex-condition}.
\end{proposition}
\begin{proof}
  Without loss of generality, assume that
   $$R(\cdot,i)\geq \theta I,\quad \text{for some } \theta>0, \quad i\in\mathcal{S},$$
   and set
   $$X^{0}(\cdot)\equiv X^{0}(\cdot;0,0,i,u),\quad \Theta(\cdot,i)=-R(\cdot,i)^{-1}S(\cdot,i),\quad i\in\mathcal{S}.$$
   Then by Lemma \ref{lem-3}, we have
  \begin{align*}
J^{0}(0,0,i;u(\cdot))&\geq\mathbb{E}\int_{0}^{T}\Big[
\big<Q(s,\alpha(s))X^{0}(s),X^{0}(s)\big>+2\big<S(s,\alpha(s))X^{0}(s),u(s)\big>\\
&\quad+\big<R(s,\alpha(s))u(s),u(s)\big>\Big]ds\\
&=\mathbb{E}\int_{0}^{T}\Big[
\big<\big(Q(s,\alpha(s))-S(s,\alpha(s))^{\top}R(s,\alpha(s))^{-1}S(s,\alpha(s))\big)X^{0}(s),X^{0}(s)\big>\\
&\quad+\big<R(s,\alpha(s))\big(u(s)-\Theta(s)X^{0}(s)\big),u(s)-\Theta(s)X^{0}(s)\big>\Big]ds\\
&\geq \theta\delta \mathbb{E}\int_{0}^{T}|u(s)|^{2}ds.
  \end{align*}
  We complete the proof by setting $\mu=\theta\delta>0$.
\end{proof}

\section{Closed-loop representation of optimal control}\label{section-5}
In the previous section, we have proven that the Problem (M-SLQ) is open-loop solvable under the uniform convexity condition \eqref{uniformly-convex-condition}. However, the form of the optimal control is still unknown. This section aims to provide a closed-loop representation of open-loop optimal control.
Before embarking on this program, we first introduce some notations. For $\mathbf{P}(\cdot)\in\mathcal{D}\left(C(0,T;\mathbb{S}^{n})\right)$ and $i\in\mathcal{S}$, we define
\begin{equation}\label{notation-MLN}
\left\{
  \begin{array}{l}
  \mathcal{M}(s,i;\mathbf{P})\triangleq P(s,i)A(s,i)+A(s,i)^{\top}P(s,i)+C(s,i)^{\top}P(s,i)C(s,i)+Q(s,i)\\[2mm]
  \qquad+\sum_{j\neq i}^{L} \pi_{ij}(s)\big\{P(s,j)-P(s,i)+[P(s,j)-P(s,i)]E_{j}(s,i)+E_{j}(s,i)^{\top}[P(s,j)-P(s,i)]\\
  \qquad+E_{j}(s,i)^{\top}P(s,j)E_{j}(s,i)\big\},\\[2mm]
  \mathcal{L}(s,i;\mathbf{P})\triangleq P(s,i)B(s,i)+C(s,i)^{\top}P(s,i)D(s,i)+S(s,i)^{\top}\\[2mm]
  \qquad +\sum_{j\neq i}^{L} \pi_{ij}(s)\big\{[P(s,j)-P(s,i)]F_{j}(s,i)+E_{j}(s,i)^{\top}P(s,j)F_{j}(s,i)\big\}\\[2mm]
   \mathcal{N}(s,i;\mathbf{P})\triangleq D(s,i)^{\top}P(s,i)D(s,i)+R(s,i)+\sum_{j\neq i}^{L} \pi_{ij}(s)F_{j}(s,i)^{\top}P(s,j)F_{j}(s,i),\quad s\in[0,T].
  \end{array}
  \right.
\end{equation}
Based on those notations, we introduce  the associated CDREs to Problem (M-SLQ) as follows:
\begin{equation}\label{CDREs}
 \left\{
 \begin{aligned}
 -\dot{P}(s,i)&=\mathcal{M}(s,i;\mathbf{P})-\mathcal{L}(s,i;\mathbf{P})\mathcal{N}(s,i;\mathbf{P})^{-1}\mathcal{L}(s,i;\mathbf{P})^{\top},\quad s\in[0,T],\\
 P(T,i)&=G(T,i),\quad i\in\mathcal{S}.
 \end{aligned}
 \right.
\end{equation}
We call a solution $\mathbf{P}(\cdot)\in\mathcal{D}\left(C(0,T;\mathbb{S}^{n})\right)$ of CDREs \eqref{CDREs} strongly regular if it satisfies the following condition
\begin{equation}\label{strongly-regular-condition}
\mathcal{N}(\cdot,i;\mathbf{P})\gg 0,\quad \forall i\in\mathcal{S}.
\end{equation}
Additionally, if the CDREs \eqref{CDREs} admits a strongly regular solution $\mathbf{P}(\cdot)\in\mathcal{D}\left(C(0,T;\mathbb{S}^{n})\right)$, we further introduce the following linear BSDE:
\begin{equation}\label{eta}
\left\{
\begin{array}{l}
d\eta(s)=-\Big\{\big[A(s,\alpha(s))^{\top}-\mathcal{L}(s,\alpha(s);\mathbf{P})\mathcal{N}(s,\alpha(s);\mathbf{P})^{-1}B(s,\alpha(s))^{\top}\big]\eta(s)\\[2mm]
\qquad\quad+\big[C(s,\alpha(s))^{\top}-\mathcal{L}(s,\alpha(s);\mathbf{P})\mathcal{N}(s,\alpha(s);\mathbf{P})^{-1}D(s,\alpha(s))^{\top}\big]\zeta(s)\\[2mm]
\qquad\quad+\big[C(s,\alpha(s))^{\top}-\mathcal{L}(s,\alpha(s);\mathbf{P})\mathcal{N}(s,\alpha(s);\mathbf{P})^{-1}D(s,\alpha(s))^{\top}\big]P(s,\alpha(s))\sigma(s)\\[2mm]
\qquad\quad+\sum_{j=1}^{L}\lambda_{j}(s)\big[E_{j}(s,\alpha(s))^{\top}-\mathcal{L}(s,\alpha(s);\mathbf{P})\mathcal{N}(s,\alpha(s);\mathbf{P})^{-1}F_{j}(s,\alpha(s))^{\top}\big]z_{j}(s)\\[2mm]
\qquad\quad+\sum_{j=1}^{L}\lambda_{j}(s)\big[E_{j}(s,\alpha(s))^{\top}-\mathcal{L}(s,\alpha(s);\mathbf{P})\mathcal{N}(s,\alpha(s);\mathbf{P})^{-1}F_{j}(s,\alpha(s))^{\top}\big]P(t,j)\gamma_{j}(s)\\[2mm]
\qquad\quad+P(s,\alpha(s))b(s)+q(s)+\sum_{j=1}^{L}\lambda_{j}(s)\big[P(s,j)-P(s,\alpha(s))\big]\gamma_{j}(s)\\[2mm]
\qquad\quad-\mathcal{L}(s,\alpha(s);\mathbf{P})\mathcal{N}(s,\alpha(s);\mathbf{P})^{-1}\rho(s)\Big\}dt+\zeta(s)dW(s)+\sum_{j=1}^{L}z_{j}(s)d \widetilde{N}_{j}(s),\\[2mm]
\eta(T)=g.
\end{array}
\right.
\end{equation}

We claim that the open-loop optimal control of Problem (M-SLQ) can be represented by the solutions to \eqref{CDREs} and \eqref{eta}.

\begin{theorem}\label{thm-closed-representation}
Suppose that the assumptions (A1)-(A2) hold, and the CDREs \eqref{CDREs} admits a strongly regular solution $\mathbf{P}(\cdot)\in\mathcal{D}(C(0,T;\mathbb{S}^{n}))$. Let $(\eta(\cdot),\zeta(\cdot),\mathbf{z}(\cdot))$ be the solution of linear BSDE \eqref{eta}. Then Problem (M-SLQ) is open-loop solvable for any given initial value $(t,x,i)\in[0,T]\times \mathbb{R}^{n}\times\mathcal{S}$. In this case, the open-loop optimal control admits the following closed-loop representation:
\begin{equation}\label{closed-loop-representation}
 u^{*}(s)=-\mathcal{N}(s,\alpha(s-);\mathbf{P})^{-1}\mathcal{L}(s,\alpha(s-);\mathbf{P})^{\top}X^{*}(s-)
 -\mathcal{N}(s,\alpha(s-);\mathbf{P})^{-1}\widetilde{\rho}(s-),\quad a.s.,\text{ }a.e.\text{ }s\in[t,T],
\end{equation}
where $X^{*}(\cdot)$ is the solution to SDE:
\begin{equation}\label{state-optimal}
   \left\{
   \begin{array}{l}
   dX^{*}(s)=\Big\{\big[A(s,\alpha(s))-B(s,\alpha(s))\mathcal{N}(s,\alpha(s);\mathbf{P})^{-1}\mathcal{L}(s,\alpha(s);\mathbf{P})^{\top}\big]X^{*}(s)\\[2mm]
   \qquad\qquad-B(s,\alpha(s))\mathcal{N}(s,\alpha(s);\mathbf{P})^{-1}\widetilde{\rho}(s)+b(s)\Big\}dt\\
   \qquad\qquad+\Big\{\big[C(s,\alpha(s))-D(s,\alpha(s))\mathcal{N}(s,\alpha(s);\mathbf{P})^{-1}\mathcal{L}(s,\alpha(s);\mathbf{P})^{\top}\big]X^{*}(s)\\[2mm]
   \qquad\qquad-D(s,\alpha(s))\mathcal{N}(s,\alpha(s);\mathbf{P})^{-1}\widetilde{\rho}(s)+\sigma(s)\Big\}dW(s)\\[2mm]
   \qquad\qquad+\sum_{j=1}^{L}\Big\{\big[E_{j}(s,\alpha(s-))-F_{j}(s,\alpha(s-))\mathcal{N}(s,\alpha(s-);\mathbf{P})^{-1}\mathcal{L}(s,\alpha(s-);\mathbf{P})^{\top}\big]X^{*}(s-)\\[2mm]
   \qquad\qquad-F_{j}(s,\alpha(s))\mathcal{N}(s,\alpha(s-);\mathbf{P})^{-1}\widetilde{\rho}(s-)+\gamma_{j}(s)\Big\}d\widetilde{N}_{j}(s),\qquad s\in[t,T],\\[2mm]
   X^{*}(t)=x,\quad \alpha(t)=i.
   \end{array}
   \right.
 \end{equation}
 and
 \begin{equation}\label{widetilde-rho}
 \begin{array}{l}
  \widetilde{\rho}(s)=B(s,\alpha(s))^{\top}\eta(s)+D(s,\alpha(s))^{\top}\zeta(s)+D(s,\alpha(s))^{\top}P(s,\alpha(s))\sigma(s)+\rho(s)\\[2mm]
  \qquad\quad+\sum_{j=1}^{L}\lambda_{j}(s)F_{j}(s,\alpha(s))^{\top}\big[P(s,j)\gamma_{j}(s)+z_{j}(s)\big],\quad  a.s., \quad a.e.\text{ }s\in[t,T].
  \end{array}
  \end{equation}
  In addition, the value function is given by
  \begin{equation}\label{value-SLQ}
   \begin{array}{l}
    V(t,x,i)=\mathbb{E}\Big[\big<P(t,i)x,x\big>+2\big<\eta(t),x\big>+\int_{t}^{T}\Big\{
   \big<P(s,\alpha(s))\sigma(s),\sigma(s)\big>+2\big<\eta(s),b(s)\big>+2\big<\zeta(s),\sigma(s)\big>\\[2mm]
\qquad+\sum_{j=1}^{L}\lambda_{j}(s)\big[
    \big<P(s,j)\gamma_{j}(s),\gamma_{j}(s)\big>
    +2\big<z_{j}(s),\gamma_{j}(s)\big>\big]-\big<\mathcal{N}(s,\alpha(s);\mathbf{P})^{-1}\widetilde{\rho}(s),\widetilde{\rho}(s)\big> \Big\}ds\Big].
   \end{array}
  \end{equation}
\end{theorem}
\begin{proof}
  Let $\mathbf{P}(\cdot)$ and $(\eta(\cdot),\zeta(\cdot),\mathbf{z}(\cdot))$ be the solutions to the CDREs \eqref{CDREs} and BSDE \eqref{eta}, respectively. %For any given initial value $(t,x,i)\in[0,T]\times \mathbb{R}^{n}\times\mathcal{S}$,
  Applying It\^o's rule to $\big<P(s,\alpha(s))X(s;t,x,i,u),X(s;t,x,i,u)\big>$, we have
  \begin{equation}\label{PX-X}
    \begin{array}{l}
    \quad \mathbb{E}\big[\big<G(T,\alpha(T))X(T;t,x,i,u),X(T;t,x,i,u)\big>-\big<P(t,i)x,x\big>\big]\\[2mm]
    =\mathbb{E} \int_{t}^{T}\Big\{
    \big<\big[\dot{P}(s,\alpha(s))+\widehat{Q}(s,\alpha(s))\big]X(s),X(s)\big>
    +2\big<\widehat{S}(s,\alpha(s))X(s),u(s)\big>
    +\big<\widehat{R}(s,\alpha(s))u(s),u(s)\big>\\[2mm]
    \quad +2\big<\widehat{q}(s),X(s)\big>
    +2\big<\widehat{\rho}(s),u(s)\big>
    +\big<P(s,\alpha(s))\sigma(s),\sigma(s)\big>
    +\sum_{j=1}^{L}\lambda_{j}(s)\big<P(s,j)\gamma_{j}(s),\gamma_{j}(s)\big>
    \Big\}ds,
    \end{array}
  \end{equation}
  where %for $i\in\mathcal{S}$,
  \begin{equation}\label{QSR-qrho}
    \left\{
    \begin{array}{l}
    \widehat{Q}(s,i)=\mathcal{M}(s,i;\mathbf{P})-Q(s,i),
    \quad\widehat{S}(s,i)=\mathcal{L}(s,i;\mathbf{P})^{\top}-S(s,i),
    \quad \widehat{R}(s,i)=\mathcal{N}(s,i;\mathbf{P})-R(s,i),\\[2mm]
    \widehat{q}(s)=P(s,\alpha(s))b(s)+C(s,\alpha(s))^{\top}P(s,\alpha(s))\sigma(s)
    +\sum_{j=1}^{L}\lambda_{j}(s)\big\{\big[P(s,j)-P(s,\alpha(s))\big]\gamma_{j}(s)\\[2mm]
    \qquad\quad+E_{j}(s,\alpha(s))^{\top}P(s,j)\gamma_{j}(s)\big\},\\[2mm]
    \widehat{\rho}(s)=D(s,\alpha(s))^{\top}P(s,\alpha(s))\sigma(s)
    +\sum_{j=1}^{L}\lambda_{j}(s)F_{j}(s,\alpha(s))^{\top}P(s,j)\gamma_{j}(s)\big\},\quad s\in[t,T].
    \end{array}
    \right.
  \end{equation}
  In addition, applying It\^o's rule to $\big<\eta(s),X(s;t,x,i,u)\big>$ yields
  \begin{equation}\label{eta-X}
    \begin{array}{l}
    \quad \mathbb{E}\big[\big<g, X(T;t,x,i,u)\big>-\big<\eta(t),x\big>\big]\\[2mm]
    =\mathbb{E}\int_{t}^{T}\Big\{
    \big<B(s,\alpha(s))^{\top}\eta(s)+D(s,\alpha(s))^{\top}\zeta(s)
    +\sum_{j=1}^{L}\lambda_{j}(s)F_{j}(s,\alpha(s))^{\top}z_{j}(s),u(s)\big>\\[2mm]
   \quad +\big<\eta(s),b(s)\big>
    +\big<\zeta(s),\sigma(s)\big>
    +\sum_{j=1}^{L}\lambda_{j}(s)\big<z_{j}(s),\gamma_{j}(s)\big>
    -\big<\widehat{q}(s)+q(s),X(s)\big>\\[2mm]
   \quad +\big<\mathcal{L}(s,\alpha(s);\mathbf{P})\mathcal{N}(s,\alpha(s);\mathbf{P})^{-1}\widetilde{\rho}(s),X(s)\big>\Big\}ds.
    \end{array}
  \end{equation}
  Thus, plugging \eqref{PX-X} and \eqref{eta-X} into cost functional \eqref{cost} and noting that
  $$\dot{P}(s,\alpha(s))+\mathcal{M}(s,\alpha(s);\mathbf{P})
  =\mathcal{L}(s,\alpha(s);\mathbf{P})\mathcal{N}(s,\alpha(s);\mathbf{P})^{-1}\mathcal{L}(s,\alpha(s);\mathbf{P})^{\top},$$
 one can easily derive
  \begin{equation}
    \begin{array}{l}
    J(t,x,i;u(\cdot))=\mathbb{E}\Big[\big<P(t,i)x,x\big>+2\big<\eta(t),x\big>\\
    \qquad+\int_{t}^{T}\Big\{
    \big<\mathcal{L}(s,\alpha(s);\mathbf{P})\mathcal{N}(s,\alpha(s);\mathbf{P})^{-1}\mathcal{L}(s,\alpha(s);\mathbf{P})^{\top}X(s),X(s)\big>\\[2mm]
    \qquad +2\big<\mathcal{L}(s,\alpha(s);\mathbf{P})^{\top}X(s),u(s)\big>
    +\big<\mathcal{N}(s,\alpha(s);\mathbf{P})u(s),u(s)\big>+2\big<\widetilde{\rho}(s),u(s)\big>\\[2mm]
    \qquad +2\big<\mathcal{L}(s,\alpha(s);\mathbf{P})\mathcal{N}(s,\alpha(s);\mathbf{P})^{-1}\widetilde{\rho}(s),X(s)\big>
    +\big<P(s,\alpha(s))\sigma(s),\sigma(s)\big>+2\big<\eta(s),b(s)\big>\\[2mm]
    \qquad+2\big<\zeta(s),\sigma(s)\big>+\sum_{j=1}^{L}\lambda_{j}(s)\big[
    \big<P(s,j)\gamma_{j}(s),\gamma_{j}(s)\big>
    +2\big<z_{j}(s),\gamma_{j}(s)\big>\big]
    \Big\}ds\Big]\\[2mm]
    \quad=V(t,x,i)+\mathbb{E}\int_{t}^{T}\big<\mathcal{N}(s,\alpha(s);\mathbf{P})\big[u(s)-u^{*}(s)\big],u(s)-u^{*}(s)\big>ds\Big].
    %\quad=\mathbb{E}\Big[\big<P(t,i)x,x\big>+2\big<\eta(t),x\big>+\int_{t}^{T}\Big\{
%   \big<P(s,\alpha(s))\sigma(s),\sigma(s)\big>+2\big<\eta(s),b(s)\big>
%   -\big<\mathcal{N}(s,\alpha(s);\mathbf{P})^{-1}\widetilde{\rho}(s),\widetilde{\rho}(s)\big>\\[2mm]
%\qquad+2\big<\zeta(s),\sigma(s)\big>+\sum_{j=1}^{L}\lambda_{j}(s)\big[
%    \big<P(s,j)\gamma_{j}(s),\gamma_{j}(s)\big>
%    +2\big<z_{j}(s),\gamma_{j}(s)\big>\big]\\[2mm]
%     \qquad+\big<\mathcal{N}(s,\alpha(s);\mathbf{P})\big[u(s)-u^{*}(s)\big],u(s)-u^{*}(s)\big>
%   %\big[u(s)+\mathcal{N}(s,\alpha(s);\mathbf{P})
%%   \mathcal{L}(s,\alpha(s);\mathbf{P})^{\top}X(s)+\mathcal{N}(s,\alpha(s);\mathbf{P})\widetilde{\rho}(s)\big],
%%   u(s)+\mathcal{N}(s,\alpha(s);\mathbf{P})
%%   \mathcal{L}(s,\alpha(s);\mathbf{P})^{\top}X(s)+\mathcal{N}(s,\alpha(s);\mathbf{P})\widetilde{\rho}(s)\big>\\[2mm]
%    \Big\}ds\Big].
    \end{array}
  \end{equation}
  Since the solution $\mathbf{P}(\cdot)$ to \eqref{CDREs} satisfying $\mathcal{N}(\cdot;P,i)\gg 0$ for any $i\in\mathcal{S}$, the desired results \eqref{closed-loop-representation} and \eqref{value-SLQ} follow from the above equation immediately.
\end{proof}
%{\color{red}
\begin{remark}\label{rmk-3}\rm
It follows from the above theorem that if CDREs \ref{CDREs}  admits a strongly regular solution, then Problem (M-SLQ) is necessarily open-loop solvable. Furthermore, by Theorem \ref{thm-open-loop-solvability}, the optimality system (combined with \eqref{FBSDE} and \eqref{stationary-condition}) is solvable. In fact, let $X^{*}(\cdot)$ and $u^{*}(\cdot)$ satisfy \eqref{state-optimal} and \eqref{closed-loop-representation} respectively, and set
\begin{align*}
&Y^{*}(s)=P(s,\alpha(s))X^{*}(s)+\eta(s),\\
&Z^{*}(s)=P(s,\alpha(s))\left[C(s,\alpha(s))X^{*}(s)+D(s,\alpha(s))u^{*}(s)+\sigma(s)\right]+\zeta(s),\\
&\Gamma_{j}^{*}(s)=P(j)\left[E_{j}(s,\alpha(s-))X^{*}(s-)+F_{j}(s,\alpha(s-))u^{*}(s)+\gamma_{j}(s)\right]\\
&\qquad \qquad +\big[P(s,j)-P(s,\alpha(s-))\big]X^{*}(s-)+z_{j}(s),
\end{align*}
where $\mathbf{P}(\cdot)$ and $(\eta(\cdot),\zeta(\cdot),\mathbf{z}(\cdot))$ are the solutions to CDREs \eqref{CDREs} and BSDE \eqref{eta}.  Then, one can further verify that the 5-tuple $(X^{*}(\cdot),Y^{*}(\cdot),Z^{*}(\cdot),\mathbf{\Gamma}^{*}(\cdot),u^{*}(\cdot))$ forms a solution to the optimality system of Problem (M-SLQ).
\end{remark}%}

\begin{remark}\label{rmk-2}\rm
If $E_{j}(\cdot,i)\equiv 0$, $F_{j}(\cdot,i)\equiv 0$, and $\gamma_{j}(\cdot)\equiv 0$ for all $i,j\in\mathcal{S}$, then the state equation \eqref{state} reduces to \eqref{MRSDS}. In this case, the notation \eqref{notation-MLN} will be simplified as following (noting that $\sum_{j=1}^{L}\pi_{ij}(s)=0$ for all $i\in\mathcal{S}$):
\begin{equation}\label{notation-MLN-2}
\left\{
  \begin{array}{l}
\tilde{\mathcal{M}}(s,i;\mathbf{P})=P(s,i)A(s,i)+A(s,i)^{\top}P(s,i)+C(s,i)^{\top}P(s,i)C(s,i)+Q(s,i)\\[2mm]
\qquad\qquad\qquad+\sum_{j=1}^{L} \pi_{ij}(s) P(s,j),\\[2mm]
\tilde{\mathcal{L}}(s,i;\mathbf{P})= P(s,i)B(s,i)+C(s,i)^{\top}P(s,i)D(s,i)+S(s,i)^{\top}\\[2mm]
\tilde{\mathcal{N}}(s,i;\mathbf{P})= D(s,i)^{\top}P(s,i)D(s,i)+R(s,i),\quad s\in[0,T],\quad i\in\mathcal{S},
  \end{array}
  \right.
\end{equation}
%which implies that the CDREs \eqref{CDREs} and BSDE \eqref{eta} reduce to those studied in \cite{Zhang-Li-Xiong-2021}.
%{\color{red}
Obviously, the CDREs derived from the SLQ with state equation \eqref{MRSDS}
\begin{equation}\label{CDREs-2}
\left\{
 \begin{aligned}
 -\dot{P}(s,i)&=\tilde{\mathcal{M}}(s,i;\mathbf{P})-\tilde{\mathcal{L}}(s,i;\mathbf{P})\tilde{\mathcal{N}}(s,i;\mathbf{P})^{-1}\tilde{\mathcal{L}}(s,i;\mathbf{P})^{\top},\quad s\in[0,T],\\
 P(T,i)&=G(T,i),\quad i\in\mathcal{S},
 \end{aligned}
 \right.
\end{equation}
is only coupled by the term $\sum_{j=1}^{L} \pi_{ij}(s) P(s,j)$. As we can see from Moon \cite{Moon-2019} and our previous work \cite{My-paper-2024-finite-M-ZLQ}, the coupling term $\sum_{j=1}^{L} \pi_{ij}(s) P(s,j)$ can be eliminated by using the dimension expansion technique to introduce a new variable
$$\mathbf{P}(\cdot)\triangleq diag(P(\cdot,1),P(\cdot,2),\cdots,P(\cdot,L)).$$
Consequently, the CDREs \eqref{CDREs-2} will degenerate into the classical single uncoupled matrix-valued differential Riccati equation. However, in this paper, the introduction of the Markovian jumps \eqref{Markovian-jumps} into the state equation makes the corresponding CDREs \eqref{CDREs} highly coupled, rendering it difficult to degenerate into a classical single uncoupled matrix-valued differential Riccati equation via the dimension expansion method. Clearly, the solvability of CDREs \eqref{CDREs} will face greater challenges. In the next section, we will provide a systematic analysis of the solvability of the CDREs \eqref{CDREs}.
%}
%Obviously, the CDREs \eqref{CDREs} examined in this paper are more complex than those in \cite{Zhang-Li-Xiong-2021}, which further compounds the challenge of analyzing their solvability. In the next section, we will provide a systematic analysis of the solvability of the CDREs \eqref{CDREs}.
\end{remark}

Clearly, if $b(\cdot)=\sigma(\cdot)=q(\cdot)=0$, $\rho(\cdot)=0$, $g=0$ and $\gamma_{j}(\cdot)=0$ for $j\in\mathcal{S}$, then the unique solution to BSDE \eqref{eta} becomes $(\eta(\cdot),\zeta(\cdot),\mathbf{z}(\cdot))=(0,0,\mathbf{0})$. Consequently, we have the following result.

\begin{corollary}\label{coro-closed-representation-0}
Suppose that the assumptions (A1)-(A2) hold, and the CDREs \eqref{CDREs} admits a solution $\mathbf{P}(\cdot)\in\mathcal{D}(C(0,T;\mathbb{S}^{n}))$ such that $\mathcal{N}(\cdot,i;\mathbf{P})\gg 0$ for any $i\in\mathcal{S}$. Then Problem (M-SLQ)$^{0}$ is open-loop solvable for any given initial value $(t,x,i)\in[0,T]\times \mathbb{R}^{n}\times\mathcal{S}$. In this case, the open-loop optimal control admits the following closed-loop representation:
\begin{equation}\label{closed-loop-representation-0}
 u^{*}(s)=-\mathcal{N}(s,\alpha(s-);\mathbf{P})^{-1}\mathcal{L}(s,\alpha(s-);\mathbf{P})^{\top}X_{0}^{*}(s-),\quad a.s.,\text{ }a.e.\text{ }s\in[t,T],
\end{equation}
where $X_{0}^{*}(\cdot)$ is the solution to SDE:
\begin{equation}\label{state-optimal-0}
   \left\{
   \begin{array}{l}
   dX_{0}^{*}(s)=\big[A(s,\alpha(s))-B(s,\alpha(s))\mathcal{N}(s,\alpha(s);\mathbf{P})^{-1}\mathcal{L}(s,\alpha(s);\mathbf{P})^{\top}\big]X_{0}^{*}(s)dt\\[2mm]
   \qquad\qquad+\big[C(s,\alpha(s))-D(s,\alpha(s))\mathcal{N}(s,\alpha(s);\mathbf{P})^{-1}\mathcal{L}(s,\alpha(s);\mathbf{P})^{\top}\big]X_{0}^{*}(s)dW(s)\\[2mm]
   \qquad\qquad+\sum_{j=1}^{L}\big[E_{j}(s,\alpha(s-))-F_{j}(s,\alpha(s-))\mathcal{N}(s,\alpha(s-);\mathbf{P})^{-1}\mathcal{L}(s,\alpha(s-);\mathbf{P})^{\top}\big]X_{0}^{*}(s-)d\widetilde{N}_{j}(s),\\[2mm]
   X_{0}^{*}(t)=x,\quad \alpha(t)=i,
   \end{array}
   \right.
 \end{equation}
and the value function is given by
  \begin{equation}\label{value-SLQ-0}
    V^{0}(t,x,i)=\big<P(t,i)x,x\big>.
  \end{equation}
\end{corollary}

%Similar to Zhang et al. \cite{Zhang-Li-Xiong-2021}, the solution to CDREs \eqref{CDREs}  satisfying the condition $\mathcal{N}(\cdot;P,i)\gg 0$ for all $i\in\mathcal{S}$ is called the strongly regular solution. Theorem \ref{thm-closed-representation} constructs the closed-loop representation of the open-loop optimal control under the assumption that CDREs \eqref{CDREs} admit a strongly regular solution. However, whether CDREs \eqref{CDREs} admit a strongly regular solution remains an open question. As noted in Zhang et al. \cite{Zhang-Li-Xiong-2021}, the strongly regular solvability of corresponding CDREs is equivalent to the associated control problem satisfying the uniform convexity condition. In the next section, we will develop this result to Problem (M-SLQ).
%further investigate the relationship between the strongly regular solvability of CDREs \eqref{CDREs} and the uniform convexity condition \eqref{uniformly-convex-condition}.

\section{Solvability of CDREs}\label{section-6}
In this section, we will establish the equivalence between uniform convexity of the cost functional and the unique strongly regular solvability of the CDREs \eqref{CDREs}. To this end, for any $\mathbf{\Theta}(\cdot)\in \mathcal{D}\left(L^{\infty}(0,T;\mathbb{R}^{m\times n})\right)$, we  introduce the following Lyapunov equation:
\begin{equation}\label{lyapunov-1}
\left\{
\begin{array}{l}
-\dot{P}(s,i)=P(s,i)A(s,i;\mathbf{\Theta})+A(s,i;\mathbf{\Theta})^{\top}P(s,i)+C(s,i;\mathbf{\Theta})^{\top}P(s,i)C(s,i;\mathbf{\Theta})+Q(s,i;\mathbf{\Theta})\\[2mm]
\qquad +\sum_{j\neq i}^{L}\pi_{ij}(s)\big\{\big[P(s,j)-P(s,i)\big]E_{j}(s,i;\mathbf{\Theta})+E_{j}(s,i;\mathbf{\Theta})^{\top}\big[P(s,j)-P(s,i)\big]\\[2mm]
\qquad +E_{j}(s,i;\mathbf{\Theta})^{\top}P(s,j)E_{j}(s,i;\mathbf{\Theta})+P(s,j)-P(s,i)\big\},\\[2mm]
P(T,i)=G(T,i),
\end{array}
\right.
\end{equation}
where
$$
\left\{
\begin{aligned}
& Q(s,i;\mathbf{\Theta})=Q(s,i)+S(s,i)^{\top}\Theta(s,i)+\Theta(s,i)^{\top}S(s,i)+\Theta(s,i)^{\top}R(s,i)\Theta(s,i), \\
& A(s,i;\mathbf{\Theta})=A(s,i)+B(s,i)\Theta(s,i),\\
&C(s,i;\mathbf{\Theta})=C(s,i)+D(s,i)\Theta(s,i),\\
&E_{j}(s,i;\mathbf{\Theta})=E_{j}(s,i)+F_{j}(s,i)\Theta(s,i), \quad s\in[0,T],\quad i,j\in\mathcal{S}.
\end{aligned}
\right.
$$
%or equivalently,
%\begin{equation}\label{lyapunov-2}
%\left\{
%\begin{aligned}
%-\dot{P}(s,i)&=\mathcal{M}(s,i;\mathbf{P})+\mathcal{L}(s,i;\mathbf{P})\Theta(s,i)+\Theta(s,i)^{\top}\mathcal{L}(s,i;\mathbf{P})^{\top}
%+\Theta(s,i)^{\top}\mathcal{N}(s,i;\mathbf{P})\Theta(s,i),\\
%P(T,i)&=G(T,i).
%\end{aligned}
%\right.
%\end{equation}
Next, we will demonstrate that solutions to CDREs \eqref{CDREs} can be iteratively approximated by the solutions to a system of Lyapunov equations. Before embarking on this program, we initially perform the preparatory procedures.

\begin{lemma}\label{lem-1}
Let assumptions (A1)-(A2) hold and $\mathbf{\Theta}(\cdot)\in \mathcal{D}\left(L^{\infty}(0,T;\mathbb{R}^{m\times n})\right)$. Then for any $(t,x,i)\in[0,T]\times\mathbb{R}^{n}\times\mathcal{S}$ and $u(\cdot)\in\mathcal{U}[t,T]$, the following holds:
\begin{equation}\label{J-Theta}
\begin{aligned}
 J^{0}(t,x,i;\Theta(\cdot,\alpha(\cdot))X_{\Theta}^{0}(\cdot)+u(\cdot))&=\big<P(t,i)x,x\big>+\mathbb{E}\int_{t}^{T}
 \Big\{\big<\mathcal{N}(s,\alpha(s);\mathbf{P})u(s),u(s)\big>\\
 &\quad+2\big<\big[\mathcal{L}(s,\alpha(s);\mathbf{P})^{\top}+\mathcal{N}(s,\alpha(s);\mathbf{P})\Theta(s,\alpha(s))\big]X_{\Theta}^{0}(s),u(s)\big>\Big\}
 \end{aligned}
\end{equation}
where $\mathbf{P}(\cdot)\in\mathcal{D}\left(C(0,T;\mathbb{S}^{n})\right)$ be the solution to Lyapunov equation \eqref{lyapunov-1} and $X_{\Theta}^{0}(\cdot)\equiv X_{\Theta}^{0}(\cdot;t,x,i,u)$ be the solution to the following SDE:
%\begin{equation}\label{state-Theta-0}
%   \left\{
%   \begin{array}{l}
%   dX_{\Theta}^{0}(s)=\Big\{\big[A(s,\alpha(s))+B(s,\alpha(s))\Theta(s,\alpha(s))\big]X_{\Theta}^{0}(s)+B(s,\alpha(s))u(s)\Big\}ds\\[2mm]
%   \qquad\qquad+\Big\{\big[C(s,\alpha(s))+D(s,\alpha(s))\Theta(s,\alpha(s))\big]X_{\Theta}^{0}(s)+D(s,\alpha(s))u(s)\Big\}dW(s)\\[2mm]
%   \qquad\qquad+\sum_{j=1}^{L}\Big\{\big[E_{j}(s,\alpha(s-))+F_{j}(s,\alpha(s-))\Theta(s,\alpha(s-))\big]X_{\Theta}^{0}(s)+F_{j}(s,\alpha(s-))u(s)\Big\}d\widetilde{N}_{j}(s),\\[2mm]
%   X_{\Theta}^{0}(t)=x,\quad \alpha(t)=i, \quad s\in[t,T].
%   \end{array}
%   \right.
% \end{equation}
\begin{equation}\label{state-Theta-0}
   \left\{
   \begin{array}{l}
   dX_{\Theta}^{0}(s)=\big[A(s,\alpha(s);\mathbf{\Theta})X_{\Theta}^{0}(s)+B(s,\alpha(s))u(s)\big]ds\\[2mm]
    \qquad\qquad+\big[C(s,\alpha(s);\mathbf{\Theta})X_{\Theta}^{0}(s)+D(s,\alpha(s))u(s)\big]dW(s)\\[2mm]
   \qquad\qquad+\sum_{j=1}^{L}\big[E_{j}(s,\alpha(s-);\mathbf{\Theta})X_{\Theta}^{0}(s)+F_{j}(s,\alpha(s-))u(s)\big]d\widetilde{N}_{j}(s),\\[2mm]
   X_{\Theta}^{0}(t)=x,\quad \alpha(t)=i, \quad s\in[t,T].
   \end{array}
   \right.
 \end{equation}
\end{lemma}
\begin{proof}
Clearly, let $S(s,i;\mathbf{\Theta})=S(s,i)+R(s,i)\Theta(s,i)$ for all $i\in\mathcal{S}$. Then
%\begin{equation}\label{cost-Theta-0}
\begin{align*}
   &\quad  J^{0}(t,x,i;\Theta(\cdot,\alpha(\cdot))X_{\Theta}^{0}(\cdot)+u(\cdot))\\
    &= \mathbb{E}\Big\{\int_{t}^{T}
    \left<
    \left(
    \begin{matrix}
    Q(s,\alpha(s)) & S(s,\alpha(s))^{\top}  \\
    S(s,\alpha(s)) & R(s,\alpha(s))
    \end{matrix}
    \right)
    \left(
    \begin{matrix}
    X_{\Theta}^{0}(s) \\
    \Theta(s,\alpha(s))X_{\Theta}^{0}(s)+u(s)
    \end{matrix}
   \right),
    \left(
    \begin{matrix}
    X_{\Theta}^{0}(s) \\
    \Theta(s,\alpha(s))X_{\Theta}^{0}(s)+u(s)
    \end{matrix}
    \right)
    \right>
    ds\\
   &\quad +\big<G(T,\alpha(T))X_{\Theta}^{0}(T), X_{\Theta}^{0}(T)\big>\Big\}\\
   &=\mathbb{E}\Big\{\int_{t}^{T}
    \left<
    \left(
    \begin{matrix}
    Q(s,\alpha(s);\mathbf{\Theta}) & S(s,\alpha(s);\mathbf{\Theta})^{\top}  \\
    S(s,\alpha(s);\mathbf{\Theta}) & R(s,\alpha(s))
    \end{matrix}
    \right)
    \left(
    \begin{matrix}
    X_{\Theta}^{0}(s) \\
 u(s)
    \end{matrix}
   \right),
    \left(
    \begin{matrix}
    X_{\Theta}^{0}(s) \\
    u(s)
    \end{matrix}
    \right)
    \right>
    ds\\
    &\quad+\big<P(T,\alpha(T))X_{\Theta}^{0}(T), X_{\Theta}^{0}(T)\big>\Big\},
  \end{align*}
%\end{equation}
 Consequently, we can complete the proof by applying It\^o's rule to $\big<P(s,\alpha(s))X_{\Theta}^{0}(s),X_{\Theta}^{0}(s)\big>$ and substituting it into the above equation.
\end{proof}

 The following result provides an estimate for the solution of the Lyapunov equation \eqref{lyapunov-1}, which plays a crucial role in establishing the equivalence between uniform convexity of the cost functional and the strongly regular solution of the CDREs \eqref{CDREs}.

\begin{proposition}\label{lem-2}
  Let assumptions (A1)-(A2) and \eqref{uniformly-convex-condition} hold. Then for any $$\mathbf{\Theta}(\cdot)\in \mathcal{D}\left(L^{\infty}(0,T;\mathbb{R}^{m\times n})\right),$$ the solution $\mathbf{P}(\cdot)\in\mathcal{D}\left(C(0,T;\mathbb{S}^{n})\right)$  to the Lyapunov equation \eqref{lyapunov-1} satisfies
  \begin{equation}\label{lyapunov-condition}
    \mathcal{N}(s,i;\mathbf{P})\geq \mu I \quad \text{and} \quad P(s,i)\geq \gamma I,\quad s\in[0,T],\quad i\in\mathcal{S},
  \end{equation}
  where $\mu$ and $\gamma$ appear in \eqref{uniformly-convex-condition} and \eqref{value-0-property}, respectively.
\end{proposition}
\begin{proof}
  By \eqref{uniformly-convex-condition} and Lemma \ref{lem-0}, we have
 $$J^{0}(t,0,i;u(\cdot))
\geq \mu \mathbb{E}\int_{t}^{T} |u(s)|^{2}ds,\quad \forall (t,i,u(\cdot))\in[0,T]\times\mathcal{S}\times\mathcal{U}[t,T].$$
For any given $\mathbf{\Theta}(\cdot)\in \mathcal{D}\left(L^{\infty}(t,T;\mathbb{R}^{m\times n})\right)$, let $\mathbf{P}(\cdot)\in\mathcal{D}\left(C(t,T;\mathbb{S}^{n})\right)$
be the solutions to Lyapunov equation \eqref{lyapunov-1} and $X_{\Theta}^{0}(\cdot)$ be the solution to \eqref{state-Theta-0} with initial value $X_{\Theta}^{0}(t)=0$.
Then by Lemma \ref{lem-1}, we have
\begin{align*}
& \mu\mathbb{E}\int_{t}^{T}\big|\Theta(s,\alpha(s))X_{\Theta}^{0}(s)+u(s)\big|^{2}ds
\leq \mathbb{E}\int_{t}^{T}
 \Big\{\big<\mathcal{N}(s,\alpha(s);\mathbf{P})u(s),u(s)\big>\\
 &\qquad+2\big<\big[\mathcal{L}(s,\alpha(s);\mathbf{P})^{\top}
 +\mathcal{N}(s,\alpha(s);\mathbf{P})\Theta(s,\alpha(s))\big]X_{\Theta}^{0}(s),u(s)\big>\Big\},\quad \forall u(\cdot)\in\mathcal{U}[0,T],
\end{align*}
which implies that %for any $u(\cdot)\in\mathcal{U}[0,T]$,
\begin{equation}\label{induction-condition-1}
 \begin{aligned}
 &\mathbb{E}\int_{t}^{T}
 \Big\{2\big<\big[\mathcal{L}(s,\alpha(s);\mathbf{P})^{\top}
 +\big(\mathcal{N}(s,\alpha(s);\mathbf{P})-\mu I\big)\Theta(s,\alpha(s))\big]X_{\Theta}^{0}(s),u(s)\big>\\
 &\quad+\big<\big[\mathcal{N}(s,\alpha(s);\mathbf{P})-\mu I\big]u(s),u(s)\big>\Big\}ds
 \geq \mu\mathbb{E}\int_{t}^{T}\big|\Theta(s,\alpha(s))X_{\Theta}^{0}(s)\big|^{2}ds\geq 0, \quad \forall u(\cdot)\in\mathcal{U}[t,T].
 \end{aligned}
\end{equation}
%Let $\Phi_{\Theta}(\cdot;t,i)$ be the solution to the following SDE:
%\begin{equation}\label{state-Phi-0}
%   \left\{
%   \begin{array}{l}
%   d\Phi_{\Theta}(s)=\big[A(s,\alpha(s))+B(s,\alpha(s))\Theta(s,\alpha(s))\big]\Phi_{\Theta}(s)ds\\[2mm]
%   \qquad\qquad+\big[C(s,\alpha(s))+D(s,\alpha(s))\Theta(s,\alpha(s))\big]\Phi_{\Theta}(s)dW(s)\\[2mm]
%   \qquad\qquad+\sum_{j=1}^{L}\big[E_{j}(s,\alpha(s-))+F_{j}(s,\alpha(s-))\Theta(s,\alpha(s-))\big]\Phi_{\Theta}(s-)d\widetilde{N}_{j}(s),\\[2mm]
%   \Phi_{\Theta}(t)=I,\quad \alpha(t)=i, \quad s\in[t,T].
%   \end{array}
%   \right.
% \end{equation}
Now, for any $\upsilon\in\mathbb{R}^{m}$, let $u(s)=\upsilon I_{[t,t+h]}(s)$ with $0\leq t< t+h\leq T$. Then by the boundedness of $\mathbf{P}(\cdot)$ and $\mathbf{\Theta}(\cdot)$, the equation \eqref{induction-condition-1} yields
\begin{equation}\label{induction-condition-2}
 \begin{aligned}
 &\quad\int_{t}^{t+h}
 \Big\{K_{1}|\upsilon||\mathbb{E}[X_{\Theta}^{0}(s)]|+\mathbb{E}\int_{t}^{t+h}\big<\big[\mathcal{N}(s,\alpha(s);\mathbf{P})-\mu I\big]\upsilon,\upsilon\big>\Big\}ds\geq 0, \quad \text{for some } K_{1}>0.
 \end{aligned}
\end{equation}
Note that for $s\in [t,t+h]$,
\begin{align*}
\big|\mathbb{E}[X_{\Theta}^{0}(s)]\big|&\leq \mathbb{E}\int_{t}^{s}\Big\{\big|
\big[A(r,\alpha(r))+B(r,\alpha(r))\Theta(r,\alpha(r))\big]X_{\Theta}^{0}(r)+B(r,\alpha(r))\upsilon
\big|ds\Big\}\\
&\leq \int_{t}^{s} K_{2}\big[\big|\mathbb{E}X_{\Theta}^{0}(r) \big|+\big|\upsilon\big|\big]dr,\quad \text{for some } K_{2}>0.
\end{align*}
By Gronwall's inequality, we obtain
$$\big|\mathbb{E}[X_{\Theta}^{0}(s)]\big|\leq \int_{t}^{s}K_{2}\big|\upsilon\big| e^{K_{2}(s-r)}dr,\quad s\in[t,t+h].$$
Substituting the above result into \eqref{induction-condition-2}, we have
 \begin{equation}
 \begin{aligned}
 &\quad\int_{t}^{t+h}
 \Big\{K_{1}K_{2}\big|\upsilon\big|^{2} \int_{t}^{s} e^{K_{2}(s-r)}dr+\mathbb{E}\int_{t}^{t+h}\big<\big[\mathcal{N}(s,\alpha(s);\mathbf{P})-\mu I\big]\upsilon,\upsilon\big>\Big\}ds\geq 0.
 \end{aligned}
\end{equation}
Dividing both sides of above by $h$ and letting $h\downarrow 0$, we obtain
$$\big<\big[\mathcal{N}(t;P,i)-\mu I\big]\upsilon,\upsilon\big>\geq 0,\quad \forall \upsilon\in\mathbb{R}^{m},$$
which is equivalent to the first inequality in \eqref{lyapunov-condition}.

On the other hand, by Theorem \ref{thm-open-loop-solvability-uniform-convex} and Lemma \ref{lem-1}, one has
$$
\big<P(t,i)x,x\big>=J^{0}(t,x,i;\Theta(\cdot,\cdot)X(\cdot))\geq V^{0}(t,x,i)\geq \gamma |x|^{2},\quad \forall (t,x,i)\in[0,T]\times\mathbb{R}^{n}\times\mathcal{S}.
$$
The second inequality in \eqref{lyapunov-condition} follows directly from the above equation.
\end{proof}

Now, for $\mathcal{A}(\cdot,i),\text{ }\mathcal{C}(\cdot,i), \text{ }\mathcal{E}_{j}(\cdot,i)\in L^{\infty}(0,T;\mathbb{R}^{n\times n})$, $\mathcal{Q}(\cdot,i)\in L^{\infty}(0,T;\mathbb{S}^{n})$ and $\mathcal{G}(T,i)\in \mathbb{S}^{n}$, we consider the linear coupled ordinary differential equations (ODEs, for short):
\begin{equation}\label{M0-ODE}
  \left\{
  \begin{array}{l}
  -\dot{\Delta}(s,i)=\Delta(s,i)\mathcal{A}(s,i)+\mathcal{A}(s,i)^{\top}\Delta(s,i)
  +\mathcal{C}(s,i)^{\top}\Delta(s,i)\mathcal{C}(s,i)+\mathcal{Q}(s,i)\\[2mm]
  \qquad\qquad\quad +\sum_{j\neq i}^{L}\pi_{ij}(s)\big\{\big[\Delta(s,j)-\Delta(s,i)\big]\mathcal{E}_{j}(s,i)
  +\mathcal{E}_{j}(s,i)^{\top}\big[\Delta(s,j)-\Delta(s,i)\big]\\[2mm]
  \qquad\qquad\quad+\mathcal{E}_{j}(s,i)^{\top}\Delta(s,j)\mathcal{E}_{j}(s,i)+\Delta(s,j)-\Delta(s,i)\big\}\\[2mm]
  \Delta(T,i)=\mathcal{G}(T,i),\quad i\in\mathcal{S}.
  \end{array}
  \right.
\end{equation}
The following result provides a comparison principle for ODEs \eqref{M0-ODE}.
\begin{proposition}\label{prop-M0}
If
 $$\mathcal{G}(T,i)\geq 0,\quad \mathcal{Q}(\cdot,i)\geq 0,\quad \forall i\in\mathcal{S},$$
Then, the unique solution $\mathbf{\Delta}(\cdot)\in \mathcal{D}\left(C(0,T;\mathbb{S}^{n})\right)$  to \eqref{M0-ODE} satisfies $\Delta(\cdot,i)\geq 0$ for all $i\in\mathcal{S}$.
\end{proposition}

%{\color{red}
\begin{proof}
Let $\Phi(\cdot)\equiv\Phi(\cdot;t,i)$ denote the solution to the following stochastic differential equation:
\begin{equation*}\label{Phi-M0}
 \left\{
 \begin{array}{l}
 d\Phi(s)=\mathcal{A}(s,\alpha(s))\Phi(s)\,ds+\mathcal{C}(s,\alpha(s))\Phi(s)\,dW(s)
 +\sum_{j=1}^{L}\mathcal{E}_{j}(s,\alpha(s-))\Phi(s-)\,d\widetilde{N}_{j}(s),\quad s\in[t,T],\\[2mm]
   \Phi(t)=I_{n},\quad \alpha(t)=i.
 \end{array}
 \right.
\end{equation*}
Applying It\^o's formula to $\big\langle\Delta(s,\alpha(s))\Phi(s),\Phi(s)\big\rangle$ and noting that
 $$
  \begin{array}{l}
  d\Delta(s,\alpha(s))=\Big\{\dot{\Delta}(s,\alpha(s))+\sum_{j=1}^{L}\lambda_{j}(s)
  \big[\Delta(s,j)-\Delta(s,\alpha(s))\big]\Big\}ds\\[2mm]
  \qquad\qquad\qquad+\sum_{j=1}^{L}\big[\Delta(s,j)-\Delta(s,\alpha(s-))\big]
  d\widetilde{N}_{j}(s),
  \end{array}
  $$
we obtain, for any $(t,i)\in[0,T]\times\mathcal{S}$,
\begin{equation}\label{eq-Delta}
\begin{array}{l}
\quad \mathbb{E}\left[\big\langle\mathcal{G}(T,\alpha(T))\Phi(T),\Phi(T)\big\rangle-\Delta(t,i)\right] \\[2mm]
= \mathbb{E}\int_{t}^{T}\Big\{\big\langle\big[\dot{\Delta}(s,\alpha(s))+\sum_{j=1}^{L}\lambda_{j}(s)
\big(\Delta(s,j)-\Delta(s,\alpha(s))\big)\big]\Phi(s)+\Delta(s,\alpha(s))\mathcal{A}(s,\alpha(s))\Phi(s) \\[2mm]
\quad +\sum_{j=1}^{L}\lambda_{j}(s)\big[\Delta(s,j)-\Delta(s,\alpha(s))\big]\mathcal{E}_{j}(s,\alpha(s))\Phi(s),\Phi(s)\big\rangle \\[2mm]
\quad +\big\langle\Delta(s,\alpha(s))\Phi(s),\mathcal{A}(s,\alpha(s))\Phi(s)\big\rangle
+\big\langle\Delta(s,\alpha(s))\mathcal{C}(s,\alpha(s))\Phi(s),\mathcal{C}(s,\alpha(s))\Phi(s)\big\rangle \\[2mm]
\quad +\sum_{j=1}^{L}\lambda_{j}(s)\big\langle\big[\Delta(s,j)-\Delta(s,\alpha(s))\big]\Phi(s)
+\Delta(s,j)\mathcal{E}_{j}(s,\alpha(s))\Phi(s),\mathcal{E}_{j}(s,\alpha(s))\Phi(s)\big\rangle\Big\}ds \\[2mm]
= \mathbb{E}\int_{t}^{T}\Big\langle\Big\{\dot{\Delta}(s,\alpha(s))+\Delta(s,\alpha(s))\mathcal{A}(s,\alpha(s))
+\mathcal{A}(s,\alpha(s))^{\top}\Delta(s,\alpha(s))
+\mathcal{C}(s,\alpha(s))^{\top}\Delta(s,\alpha(s))\mathcal{C}(s,\alpha(s)) \\[2mm]
\quad+\sum_{j=1}^{L}\lambda_{j}(s)\Big[\big(\Delta(s,j)-\Delta(s,\alpha(s))\big)\mathcal{E}_{j}(s,\alpha(s))
+\mathcal{E}_{j}(s,\alpha(s))^{\top}\big(\Delta(s,j)-\Delta(s,\alpha(s))\big) \\[2mm]
\quad+\mathcal{E}_{j}(s,\alpha(s))^{\top}\Delta(s,j)\mathcal{E}_{j}(s,\alpha(s))+\Delta(s,j)-\Delta(s,\alpha(s))\Big]\Big\}\Phi(s),\Phi(s)
\Big\rangle ds.
\end{array}
\end{equation}
Observe that equation \eqref{M0-ODE} implies
$$
\begin{array}{l}
\dot{\Delta}(s,\alpha(s))+\Delta(s,\alpha(s))\mathcal{A}(s,\alpha(s))
+\mathcal{A}(s,\alpha(s))^{\top}\Delta(s,\alpha(s))
+\mathcal{C}(s,\alpha(s))^{\top}\Delta(s,\alpha(s))\mathcal{C}(s,\alpha(s)) \\[2mm]
+\sum_{j=1}^{L}\lambda_{j}(s)\Big[\big(\Delta(s,j)-\Delta(s,\alpha(s))\big)\mathcal{E}_{j}(s,\alpha(s))
+\mathcal{E}_{j}(s,\alpha(s))^{\top}\big(\Delta(s,j)-\Delta(s,\alpha(s))\big) \\[2mm]
+\mathcal{E}_{j}(s,\alpha(s))^{\top}\Delta(s,j)\mathcal{E}_{j}(s,\alpha(s))+\Delta(s,j)-\Delta(s,\alpha(s))\Big] = -\mathcal{Q}(s,\alpha(s)).
\end{array}
$$
Substituting this into equation \eqref{eq-Delta} yields
$$
\begin{array}{l}
\quad \mathbb{E}\left[\big\langle\mathcal{G}(T,\alpha(T))\Phi(T),\Phi(T)\big\rangle-\Delta(t,i)\right]
= -\mathbb{E}\int_{t}^{T}\big\langle\mathcal{Q}(s,\alpha(s))\Phi(s),\Phi(s)\big\rangle ds,
\end{array}
$$
which further implies that
\begin{equation*}
\Delta(t,i)=\mathbb{E}\Big[\Phi(T;t,i)^{\top}\mathcal{G}(T,\alpha(T))\Phi(T;t,i)
+\int_{t}^{T}\Phi(s;t,i)^{\top}\mathcal{Q}(s,\alpha(s))\Phi(s;t,i)\,ds\;\big|\;\alpha(t)=i \Big].
\end{equation*}
The desired result then follows directly from the above identity.
\end{proof}
%}

Now we prove the equivalence between the uniform convexity of the cost functional and the strongly regular solvability of the Riccati equation.

\begin{theorem}\label{thm-CDRE-uniform-convex}
 Let assumptions (A1)-(A2) hold. Then the condition \eqref{uniformly-convex-condition} holds if and only if the CDREs \eqref{CDREs} admits a solution $\mathbf{P}(\cdot)\in\mathcal{D}\left(C(0,T;\mathbb{S}^{n})\right)$  such that $\mathcal{N}(\cdot,i;\mathbf{P})\gg 0$ for any $i\in\mathcal{S}$.
\end{theorem}

\begin{proof}
\textbf{Necessity.} Let $\mathbf{P^{(0)}}(\cdot)\in\mathcal{D}\left(C(0,T;\mathbb{S}^{n})\right)$ be the solution of
\begin{equation}\label{P0}
\left\{
\begin{aligned}
-\dot{P}^{(0)}(s,i)&=\mathcal{M}(s,i;\mathbf{P^{(0)}}),\quad s\in[0,T],\\
P^{(0)}(T,i)&=G(T,i).
\end{aligned}
\right.
\end{equation}
Then by Proposition \ref{lem-2} with $\mathbf{\Theta}(\cdot)\equiv \mathbf{0}$, we have
$$\mathcal{N}(s,i;\mathbf{P^{(0)}})\geq \mu I \quad \text{and} \quad P^{(0)}(s,i)\geq \gamma I,\quad s\in[0,T],\quad i\in\mathcal{S}.$$
Now, for $k=0,1,2,\cdots$ and $s\in [0,T]$, we defined
 $$\left\{
 \begin{array}{l}
 \Theta^{(k)}(s,i)=-\mathcal{N}(s,i;\mathbf{P^{(k)}})^{-1}\mathcal{L}(s,i;\mathbf{P^{(k)}})^{\top},\quad
 A^{(k)}(s,i)=A(s,i)+B(s,i)\Theta^{(k)}(s,i),\\[2mm]
 C^{(k)}(s,i)=C(s,i)+D(s,i)\Theta^{(k)}(s,i),\quad
E_{j}^{(k)}(s,i)=E_{j}(s,i)+F_{j}(s,i)\Theta^{(k)}(s,i),\\[2mm]
Q^{(k)}(s,i)=Q(s,i)+S(s,i)^{\top}\Theta^{(k)}(s,i)+\Theta^{(k)}(s,i)^{\top}S(s,i)
+\Theta^{(k)}(s,i)^{\top}R(s,i)\Theta^{(k)}(s,i),
 \end{array}
 \right.
 $$
and let $\mathbf{P^{(k+1)}}(\cdot)$ be the solution of
\begin{equation}\label{PK}
  \left\{
  \begin{array}{l}
-\dot{P}^{(k+1)}(s,i)=P^{(k+1)}(s,i)A^{(k)}(s,i)+A^{(k)}(s,i)^{\top}P^{(k+1)}(s,i)
+C^{(k)}(s,i)^{\top}P^{(k+1)}(s,i)C^{(k)}(s,i)\\[2mm]
\quad +\sum_{j\neq i}^{L}\pi_{ij}(s)\big\{\big[P^{(k+1)}(s,j)-P^{(k+1)}(s,i)\big]E_{j}^{(k)}(s,i)
+E_{j}^{(k)}(s,i)^{\top}\big[P^{(k+1)}(s,j)-P^{(k+1)}(s,i)\big]\\[2mm]
\quad +E_{j}^{(k)}(s,i)^{\top}P^{(k+1)}(s,j)E_{j}^{(k)}(s,i)+P^{(k+1)}(s,j)-P^{(k+1)}(s,i)\big\}+Q^{(k)}(s,i),\\[2mm]
P^{(k+1)}(T,i)=G(T,i),\quad k=0,1,2,\cdots .
\end{array}
  \right.
\end{equation}
Then by Proposition \ref{lem-2} again, we have
$$\mathcal{N}(s,i;\mathbf{P^{(k+1)}})\geq \mu I \quad \text{and} \quad P^{(k+1)}(s,i)\geq \gamma I,\quad s\in[0,T],\quad i\in\mathcal{S},\quad k=0,1,2,\cdots.$$

Next, we will prove that for any $i\in\mathcal{S}$, $\{P^{(k)}(\cdot,i)\}_{k=1}^{\infty}$ uniformly converges in $C(0,T;\mathbb{S}^{n})$. To this end, for $ k=1,2,\cdots,$ we let
$$\Delta^{(k)}(s,i)\triangleq P^{(k)}(s,i)-P^{(k+1)}(s,i),\quad \Lambda^{(k)}(s,i)\triangleq \Theta^{(k-1)}(s,i)-\Theta^{(k)}(s,i),\quad s\in[0,T], \quad i\in\mathcal{S}.$$
Then we have
\begin{align*}
&-\dot{\Delta}^{(k)}(s,i)=P^{(k)}(s,i)A^{(k-1)}(s,i)+A^{(k-1)}(s,i)^{\top}P^{(k)}(s,i)+C^{(k-1)}(s,i)^{\top}P^{(k)}(s,i)C^{(k-1)}(s,i)\\
&\quad+\sum_{j\neq i}^{L}\pi_{ij}(s)\big\{\big[P^{(k)}(s,j)-P^{(k)}(s,i)\big]E_{j}^{(k-1)}(s,i)+E_{j}^{(k-1)}(s,i)^{\top}\big[P^{(k)}(s,j)-P^{(k)}(s,i)\big]\\
&\quad +E_{j}^{(k-1)}(s,i)^{\top}P^{(k)}(s,j)E_{j}^{(k-1)}(s,i)+P^{(k)}(s,j)-P^{(k)}(s,i)\big\} +Q^{(k-1)}(s,i)\\
&\quad-P^{(k+1)}(s,i)A^{(k)}(s,i)-A^{(k)}(s,i)^{\top}P^{(k+1)}(s,i)-C^{(k)}(s,i)^{\top}P^{(k+1)}(s,i)C^{(k)}(s,i)-Q^{(k)}(s,i)\\
&\quad-\sum_{j\neq i}^{L}\pi_{ij}(s)\big\{\big[P^{(k+1)}(s,j)-P^{(k+1)}(s,i)\big]E_{j}^{(k)}(s,i)+E_{j}^{(k)}(s,i)^{\top}\big[P^{(k+1)}(s,j)-P^{(k+1)}(s,i)\big]\\
&\quad +E_{j}^{(k)}(s,i)^{\top}P^{(k+1)}(s,j)E_{j}^{(k)}(s,i)+P^{(k+1)}(s,j)-P^{(k+1)}(s,i)\big\}\\
&=\Delta^{(k)}(s,i)A^{(k)}(s,i)+A^{(k)}(s,i)^{\top}\Delta^{(k)}(s,i)+C^{(k)}(s,i)^{\top}\Delta^{(k)}(s,i)C^{(k)}(s,i)\\
&\quad+P^{(k)}(s,i)B(s,i)\Lambda^{(k)}(s,i)+\Lambda^{(k)}(s,i)^{\top}B(s,i)^{\top}P^{(k)}(s,i)+C^{(k-1)}(s,i)^{\top}P^{(k)}(s,i)C^{(k-1)}(s,i)\\
&\quad-C^{(k)}(s,i)^{\top}P^{(k)}(s,i)C^{(k)}(s,i) +S(s,i)^{\top}\Lambda^{(k)}(s,i)+\Lambda^{(k)}(s,i)^{\top}S(s,i)\\
&\quad+\Theta^{(k-1)}(s,i)^{\top}R(s,i)\Theta^{(k-1)}(s,i)-\Theta^{(k)}(s,i)^{\top}R(s,i)\Theta^{(k)}(s,i)\\
&\quad+\sum_{j\neq i}^{L}\pi_{ij}(s)\big\{\big[\Delta^{(k)}(s,j)-\Delta^{(k)}(s,i)\big]E_{j}^{(k)}(s,i)+E_{j}^{(k)}(s,i)^{\top}\big[\Delta^{(k)}(s,j)-\Delta^{(k)}(s,i)\big]\\
&\quad+\big[P^{(k)}(s,j)-P^{(k)}(s,i)\big]F_{j}(s,i)\Lambda^{(k)}(s,i)+\Lambda^{(k)}(s,i)^{\top}F_{j}(s,i)^{\top}\big[P^{(k)}(s,j)-P^{(k)}(s,i)\big]\\
&\quad +E_{j}^{(k-1)}(s,i)^{\top}P^{(k)}(s,j)E_{j}^{(k-1)}(s,i)-E_{j}^{(k)}(s,i)^{\top}P^{(k+1)}(s,j)E_{j}^{(k)}(s,i)+\Delta^{(k)}(s,j)-\Delta^{(k)}(s,i)\big\}.
\end{align*}

Noting that
\begin{align*}
& C^{(k-1)}(s,i)^{\top}P^{(k)}(s,i)C^{(k-1)}(s,i)-C^{(k)}(s,i)^{\top}P^{(k)}(s,i)C^{(k)}(s,i)\\
&\quad=\Lambda^{(k)}(s,i)^{\top}D(s,i)^{\top}P^{(k)}(s,i)D(s,i)\Lambda^{(k)}(s,i)
+\Lambda^{(k)}(s,i)^{\top}D(s,i)^{\top}P^{(k)}(s,i)C^{(k)}(s,i)\\
&\quad\quad+C^{(k)}(s,i)^{\top}P^{(k)}(s,i)D(s,i)\Lambda^{(k)}(s,i),\\
&\Theta^{(k-1)}(s,i)^{\top}R(s,i)\Theta^{(k-1)}(s,i)-\Theta^{(k)}(s,i)^{\top}R(s,i)\Theta^{(k)}(s,i)\\
&\quad=\Lambda^{(k)}(s,i)^{\top}R(s,i)\Lambda^{(k)}(s,i)+\Lambda^{(k)}(s,i)^{\top}R(s,i)\Theta^{(k)}(s,i)
+\Theta^{(k)}(s,i)^{\top}R(s,i)\Lambda^{(k)}(s,i),\\
&E_{j}^{(k-1)}(s,i)^{\top}P^{(k)}(s,j)E_{j}^{(k-1)}(s,i)-E_{j}^{(k)}(s,i)^{\top}P^{(k+1)}(s,j)E_{j}^{(k)}(s,i)\\
&\quad=\Lambda^{(k)}(s,i)^{\top}F_{j}(s,i)^{\top}P^{(k)}(s,i)F_{j}(s,i)\Lambda^{(k)}(s,i)
+\Lambda^{(k)}(s,i)^{\top}F_{j}(s,i)^{\top}P^{(k)}(s,i)E_{j}^{(k)}(s,i)\\
&\quad\quad+E_{j}^{(k)}(s,i)^{\top}P^{(k)}(s,i)F_{j}(s,i)\Lambda^{(k)}(s,i).
\end{align*}
further implies
\begin{equation}\label{Delta-2}
\begin{array}{l}
-\dot{\Delta}^{(k)}(s,i)=\Delta^{(k)}(s,i)A^{(k)}(s,i)+A^{(k)}(s,i)^{\top}\Delta^{(k)}(s,i)
+C^{(k)}(s,i)^{\top}\Delta^{(k)}(s,i)C^{(k)}(s,i)\\[2mm]
\quad+\sum_{j\neq i}^{L}\pi_{ij}(s)\big\{\big[\Delta^{(k)}(s,j)-\Delta^{(k)}(s,i)\big]E_{j}^{(k)}(s,i)
+E_{j}^{(k)}(s,i)^{\top}\big[\Delta^{(k)}(s,j)-\Delta^{(k)}(s,i)\big]\\[2mm]
\quad +\Delta^{(k)}(s,j)-\Delta^{(k)}(s,i)\big\}
+\Lambda^{(k)}(s,i)^{\top}\big[\mathcal{L}(s,i;\mathbf{P^{(k)}})^{\top}+\mathcal{N}(s,i;\mathbf{P^{(k)}})\Theta^{(k)}(s,i)\big]\\[2mm]
\quad +\big[\mathcal{L}(s,i;\mathbf{P^{(k)}})^{\top}+\mathcal{N}(s,i;\mathbf{P^{(k)}})\Theta^{(k)}(s,i)\big]^{\top}\Lambda^{(k)}(s,i)
+\Lambda^{(k)}(s,i)^{\top}\mathcal{N}(s,i;\mathbf{P^{(k)}})\Lambda^{(k)}(s,i)\\[2mm]
=\Delta^{(k)}(s,i)A^{(k)}(s,i)+A^{(k)}(s,i)^{\top}\Delta^{(k)}(s,i)
+C^{(k)}(s,i)^{\top}\Delta^{(k)}(s,i)C^{(k)}(s,i)\\[2mm]
\quad+\sum_{j\neq i}^{L}\pi_{ij}(s)\big\{\big[\Delta^{(k)}(s,j)-\Delta^{(k)}(s,i)\big]E_{j}^{(k)}(s,i)
+E_{j}^{(k)}(s,i)^{\top}\big[\Delta^{(k)}(s,j)-\Delta^{(k)}(s,i)\big]\\[2mm]
\quad +\Delta^{(k)}(s,j)-\Delta^{(k)}(s,i)\big\}
+\Lambda^{(k)}(s,i)^{\top}\mathcal{N}(s,i;\mathbf{P^{(k)}})\Lambda^{(k)}(s,i).
\end{array}
\end{equation}
Note that
$$\Delta(T,i)=0,\quad \text{and} \quad  \mathcal{N}(\cdot,i;\mathbf{P^{(k)}})\geq \mu I,\quad \forall i\in\mathcal{S}.$$
By Proposition \ref{prop-M0}, we have
$$\Delta(s,i)\geq 0,\quad s\in[0,T],\quad i\in\mathcal{S},$$
which implies
$$
P^{(1)}(s,i)\geq P^{(k)}(s,i)\geq P^{(k+1)}(s,i)\geq\gamma I,\quad s\in[0,T],\quad i\in\mathcal{S}.
$$
Consequently, $\{P^{(k)}(\cdot,i)\}_{k=1}^{\infty}$ is uniformly bounded for any $i\in\mathcal{S}$.

On the other hand, observe that
\begin{align*}
\Lambda^{(k)}(s,i)&=\Theta^{(k-1)}(s,i)-\Theta^{(k)}(s,i)\\
&=\mathcal{N}(s,i;\mathbf{P^{(k)}})^{-1}\mathcal{L}(s,i;\mathbf{P^{(k)}})^{\top}
-\mathcal{N}(s,i;\mathbf{P^{(k-1)}})^{-1}\mathcal{L}(s;P^{(k-1)},i)^{\top}\\
&=\big[\mathcal{N}(s,i;\mathbf{P^{(k)}})^{-1}-\mathcal{N}(s,i;\mathbf{P^{(k-1)}})^{-1}\big]\mathcal{L}(s,i;\mathbf{P^{(k)}})^{\top}\\
&\quad-\mathcal{N}(s,i;\mathbf{P^{(k-1)}})^{-1}\big[\mathcal{L}(s,i;\mathbf{P^{(k-1)}})^{\top}-\mathcal{L}(s,i;\mathbf{P^{(k)}})^{\top}\big]\\
&=\mathcal{N}(s,i;\mathbf{P^{(k)}})^{-1}\big[\mathcal{N}(s,i;\mathbf{P^{(k-1)}})
-\mathcal{N}(s,i;\mathbf{P^{(k)}})\big]\mathcal{N}(s,i;\mathbf{P^{(k-1)}})^{-1}\mathcal{L}(s,i;\mathbf{P^{(k)}})^{\top}\\
&\quad -\mathcal{N}(s,i;\mathbf{P^{(k-1)}})^{-1}\big[\mathcal{L}(s,i;\mathbf{P^{(k-1)}})-\mathcal{L}(s,i;\mathbf{P^{(k)}})\big]^{\top}\\
&=\mathcal{N}(s,i;\mathbf{P^{(k)}})^{-1}\widetilde{\mathcal{N}}(s,i;\mathbf{\Delta^{(k-1)}})\mathcal{N}(s,i;\mathbf{P^{(k-1)}})^{-1}\mathcal{L}(s,i;\mathbf{P^{(k)}})^{\top}\\
&\quad -\mathcal{N}(s,i;\mathbf{P^{(k-1)}})^{-1}\widetilde{\mathcal{L}}(s,i;\mathbf{\Delta^{(k-1)}})^{\top},
\end{align*}
where for any $i\in\mathcal{S}$ and $s\in[0,T]$,
\begin{align*}
\widetilde{\mathcal{N}}(s,i;\mathbf{\Delta^{(k-1)}})&\triangleq D(s,i)^{\top}\Delta^{(k-1)}(s,i)D(s,i)+\sum_{j\neq i}^{L} \pi_{ij}(s)F_{j}(s,i)^{\top}\Delta^{(k-1)}(s,j)F_{j}(s,i),\\
\widetilde{\mathcal{L}}(s,i;\mathbf{\Delta^{(k-1)}})&\triangleq \Delta^{(k-1)}(s,i)B(s,i)+C(s,i)^{\top}\Delta^{(k-1)}(s,i)D(s,i)
  \\
  &\quad+\sum_{j\neq i}^{L} \pi_{ij}(s)\big\{[\Delta^{(k-1)}(s,j)-\Delta^{(k-1)}(s,i)]F_{j}(s,i)+E_{j}(s,i)^{\top}\Delta^{(k-1)}(s,j)F_{j}(s,i)\big\}.
\end{align*}
Then by assumptions (A1)-(A2) and the uniform boundedness of $\{P^{(k)}(\cdot,i)\}_{k=1}^{\infty}$, there exist a sufficient large constant $M>0$ such that
\begin{align*}
 \big|\Lambda^{(k)}(s,i)^{\top}\mathcal{N}(s,i;\mathbf{P^{(k)}})\Lambda^{(k)}(s,i)\big|
 &\leq \big(|\Theta^{(k)}(s,i)|+|\Theta^{(k-1)}(s,i)|\big)\big|\mathcal{N}(s,i;\mathbf{P^{(k)}})\big|\big|\Lambda^{(k)}(s,i)\big|\\
&\leq M ||\mathbf{\Delta^{(k-1)}}(s)||,
\end{align*}
where
$$
||\mathbf{\Delta^{(k)}}(s)||\triangleq \max_{i\in\mathcal{S}}|\Delta^{(k)}(s,i)|,\quad k= 0,1,2,\cdots .
$$
Consequently, the equation \eqref{Delta-2} further yields
$$
||\mathbf{\Delta^{(k)}}(t)||\leq\int_{t}^{T}M\big[||\mathbf{\Delta^{(k)}}(s)||+||\mathbf{\Delta^{(k-1)}}(s)||\big]ds,\quad \forall s\in[0,T],\quad k\geq 1.
$$
By Gronwall's inequality, we obtain
$$
||\mathbf{\Delta^{(k)}}(t)||\leq Me^{MT}\int_{t}^{T}||\mathbf{\Delta^{(k-1)}}(s)||ds.
$$
Set
$$a\triangleq \max_{s\in[0,T]}||\mathbf{\Delta^{(0)}}(s)||.$$
By induction, one has
$$
||\mathbf{\Delta^{(k)}}(t)||\leq a\frac{(MTe^{MT})^{k}}{k!},\quad \forall t\in[0,T], \quad k=0,1,2,\cdots,
$$
which implies the uniform convergence of $\{\mathbf{P^{(k)}}(\cdot)\}_{k=0}^{\infty}$. Let $\mathbf{P}(\cdot)$ be  the limit of $\{\mathbf{P^{(k)}}(\cdot)\}_{k=0}^{\infty}$, then we have
$$
\mathcal{N}(s,i;\mathbf{P})\geq \mu I,\quad s\in[0,T],\quad i\in\mathcal{S},
$$
and as $k\rightarrow\infty$,
$$\left\{
 \begin{array}{l}
 \Theta^{(k)}(s,i)\rightarrow -\mathcal{N}(s,i;\mathbf{P})^{-1}\mathcal{L}(s,i;\mathbf{P})^{\top}\triangleq \Theta(s,i), \quad \text{in } L^{\infty}(0,T;\mathbb{R}^{m\times n}),\\[2mm]
 A^{(k)}(s,i)\rightarrow A(s,i)+B(s,i)\Theta(s,i),\quad \text{in } L^{\infty}(0,T;\mathbb{R}^{n\times n}),\\[2mm]
 C^{(k)}(s,i)\rightarrow C(s,i)+D(s,i)\Theta(s,i),\quad \text{in } L^{\infty}(0,T;\mathbb{R}^{n\times n}),\\[2mm]
E_{j}^{(k)}(s,i)\rightarrow E_{j}(s,i)+F_{j}(s,i)\Theta(s,i),\quad \text{in } L^{\infty}(0,T;\mathbb{R}^{n\times n}),\\[2mm]
Q^{(k)}(s,i)\rightarrow Q(s,i)+S(s,i)^{\top}\Theta(s,i)+\Theta(s,i)^{\top}S(s,i)
+\Theta(s,i)^{\top}R(s,i)\Theta(s,i),\quad \text{in } L^{\infty}(0,T;\mathbb{S}^{n}).
 \end{array}
 \right.
 $$
 Consequently, $\mathbf{P}(\cdot)$  satisfied the differential equation \eqref{lyapunov-1} with $\Theta(s,i)=-\mathcal{N}(s,i;\mathbf{P})^{-1}\mathcal{L}(s,i;\mathbf{P})^{\top}$, which is equivalent to CDREs \eqref{CDREs}.

\textbf{Sufficiency.} Let $\mathbf{P}(\cdot)$  be the solution of CDREs \eqref{CDREs} and satisfy the condition
$$\mathcal{N}(s,i;\mathbf{P})\geq \mu I,\quad a.e. \text{ }s\in[0,T],\quad \forall i\in\mathcal{S}.$$
Set
$$
\Theta(s,i)=-\mathcal{N}(s,i;\mathbf{P})^{-1}\mathcal{L}(s,i;\mathbf{P})^{\top}, \quad a.e. \text{ }s\in[0,T],\quad \forall i\in\mathcal{S}.
$$
Then by the proof of Theorem \ref{thm-closed-representation} and Lemma \ref{lem-3}, we have
\begin{align*}
J^{0}(0,0,i,;u(\cdot))&=\mathbb{E}\int_{0}^{T}\big<\mathcal{N}(s,\alpha(s);\mathbf{P})\big[u(s)-\Theta(s,\alpha(s))X^{0}(s;0,0,i,u)\big],u(s)\\
&\quad-\Theta(s,\alpha(s))X^{0}(s;0,0,i,u)\big>ds\\
&\geq \mu\delta\mathbb{E}\int_{0}^{T}\big<u(s),u(s)\big>ds,\quad \forall u(\cdot)\in\mathcal{U}[0,T],\quad \forall i\in\mathcal{S}.
\end{align*}
This completes the proof.
\end{proof}

\begin{remark}\rm
From the first part of the proof of Theorem \ref{thm-CDRE-uniform-convex}, we see that if \eqref{uniformly-convex-condition} holds, then the strongly regular solution of \eqref{CDREs} satisfies
$$
\mathcal{N}(s,i;\mathbf{P})\geq \mu I,
$$
with the same constant $\mu>0$.
\end{remark}

Theorem \ref{thm-CDRE-uniform-convex} establishes the existence of a strongly regular solution to CDREs \eqref{CDREs}. Combining with the Corollary \ref{coro-closed-representation-0}, we can obtain that a strongly regular solution to CDREs \eqref{CDREs} must satisfy
$$V^{0}(t,x,i)=\left<P(t,i)x,x\right>,\quad \forall (t,x,i)\in [0,T]\times\mathbb{R}^{n}\times\mathcal{S}.$$
By the unique representation of the value function for Problem (M-SLQ)$^{0}$, one can further confirm that the CDREs \eqref{CDREs} admit at most one strongly regular solution. Together with Theorem \ref{thm-closed-representation},  we have the following corollary.

\begin{corollary}
Suppose that the assumptions (A1)-(A2) and \eqref{uniformly-convex-condition} hold. Then CDREs \eqref{CDREs} admits a unique solution $\mathbf{P}(\cdot)\in\mathcal{D}(C(0,T;\mathbb{S}^{n}))$ such that $\mathcal{N}(\cdot;P,i)\gg 0$ for any $i\in\mathcal{S}$, and Problem (M-SLQ) is uniquely open-loop solvable for any given initial value $(t,x,i)\in[0,T]\times \mathbb{R}^{n}\times\mathcal{S}$. In this case, the open-loop optimal control admits the closed-loop representation \eqref{closed-loop-representation}, and the value function is given by \eqref{value-SLQ}.
\end{corollary}

\section{Application to mean-variance portfolio selection}\label{section-7}
In this section, we will solve a mean-variance (MV, for short) portfolio selection problem based on the developed results in previous sections.
%\subsection{Problem formulation}
Consider a financial market that consists of one risk-free asset (e.g., savings account) and $m$ risky assets (e.g., stocks). Suppose that the prices of the risk-free asset and risky assets are affected by the market mode $\alpha(\cdot)$ that switches among a finite number of states. More specifically, we assume that the price of the risk-free asset evolves according to
$$dS_{0}(t)=r(t,\alpha(t))S_{0}(t)dt,\qquad S_{0}(0)=1,$$
and the price processes of the risky assets are governed by %is given by
\begin{equation}\label{risky-asset-price}
\left\{
\begin{array}{l}
dS_{k}(t)=S_{k}(t-)\big[\mu_{k}(t,\alpha(t))dt+\sigma_{k}(t,\alpha(t))dW(t)
+\sum_{j=1}^{L}\gamma_{kj}(t,\alpha(t-))d\widetilde{N}_{j}(t)\big],\quad s\in[0,T],\\
S_{k}(0)=s_{k_0},\quad k=1,2,\cdots,m.
\end{array}
\right.
\end{equation}
If $\gamma_{kj}(\cdot,i)\equiv 0$ for all $k=1,2,\cdots,m$ and $i,\text{ }j\in\mathcal{S}$, then the price of risky assets will degenerate into the one studied in Zhou and Yin \cite{Zhou-2003-MV}. Here, we allow the prices of risky assets to experience a jump in growth or decline with the switching of market mode. Therefore, our model is more general than the one investigated in Zhou and Yin \cite{Zhou-2003-MV}.

Let $u(\cdot)=\big[u_{1}(\cdot),u_{2}(\cdot),\cdots,u_{m}(\cdot)\big]^{\top}$ be an investment strategy for a investor, in which $u_{k}(t)$ represents the \emph{dollar amount} invested in the $k$-th risky asset at time $t$. The set of all admissible investment strategies is denoted by $\mathcal{A}\triangleq
L_{\mathcal{P}}^{2}(0,T;\mathbb{R}^{m})$.
Thus, the wealth process $X(\cdot)$ associated with strategy $u(\cdot)$  is given by
\begin{equation}\label{wealth}
\left\{
\begin{array}{l}
dX(t)=\big[r(t,\alpha(t))X(t)+B(t,\alpha(t))u(t)\big]dt+\sigma(t,\alpha(t))u(t)dW(t)
+\sum_{j=1}^{L}\gamma_{j}(t,\alpha(t-))u(t)d\widetilde{N}_{j}(t),\\[2mm]
X(0)=x,
\end{array}
\right.
\end{equation}
where
\begin{align*}
&B(\cdot,i)=\big[\mu_{1}(\cdot,i)-r(\cdot,i),\mu_{2}(\cdot,i)-r(\cdot,i),\cdots, \mu_{m}(\cdot,i)-r(\cdot,i)\big]\in L^{\infty}(0,T;\mathbb{R}^{m}),\\
&\sigma(\cdot,i)=\big[\sigma_{1}(\cdot,i),\sigma_{2}(\cdot,i),\cdots, \sigma_{m}(\cdot,i)\big]\in L^{\infty}(0,T;\mathbb{R}^{m}),\\
&\gamma_{j}(\cdot,i)=\big[\gamma_{1j}(\cdot,i),\gamma_{2j}(\cdot,i),\cdots, \gamma_{mj}(\cdot,i)\big]\in L^{\infty}(0,T;(-1,+\infty)^{m}).
\end{align*}
Throughout this section, %we fixed the initial value of Markov chain with $\alpha(0)=i\in\mathcal{S}$ and
 denote the solution to \eqref{wealth} as $X^{u}(\cdot;x,i)$ to emphasize the dependence of the wealth process on initial values $(x,i)$ and investment strategy $u$. Clearly, given $u(\cdot)\in\mathcal{A}$ and $(x,i)\in\mathbb{R}\times \mathcal{S}$, the SDE \eqref{wealth} admits a unique solution $X^{u}(\cdot;x,i)\in L_{\mathbb{F}}^{2}(0,T;\mathbb{R})$. In the rest of this section, we assume the following uniformly nondegenerate condition always holds:
\begin{equation}\label{uniformly-nondegenerate-condition}
\begin{array}{l}
\sigma(t,i)^{\top}\sigma(t,i)+\sum_{j=1}^{L}\lambda_{j}(t)\gamma_{j}(t,i)^{\top}\gamma_{j}(t,i)\geq \delta I,\quad a.e.\text{ }t\in[0,T], \text{ }i\in\mathcal{S},\text{ for some } \delta>0.
\end{array}
\end{equation}

We now proceed to formulate the MV portfolio selection problem that is of pivotal interest to this section.

\textbf{Problem (MV):} For $(x,i)\in\mathbb{R}\times\mathcal{S}$, find an admissible portfolio $u^{*}(\cdot)\in\mathcal{A}$ such that
%Under the MV criterion, the investor seeks to solve the following portfolio selection problem:
\begin{equation}\label{MV}
 % \textbf{Problem MV: }
  \left\{
  \begin{aligned}
  &J_{MV}(x,i;u^{*}(\cdot))=\min_{u(\cdot)\in \mathcal{A}}\text{ } J_{MV}(x,i;u(\cdot)):=
  %\min_{\pi(\cdot)\in\mathcal{A}}
  \mathbb{E}\big[\big(X^{u}(T;x,i)-z\big)^{2}\big],\\
 & \text{subject to }\left\{\begin{array}{l}
  \mathbb{E}[X^{u}(T;x,i)\big]=z,\\[2mm]
  X^{u}(\cdot;x,i) \text{ is the solution to \eqref{wealth}.}
  \end{array}
  \right.
  \end{aligned}
  \right.
\end{equation}
\begin{definition}
  The Problem (MV) is called feasible if there is at least one portfolio satisfying all the constraints. The problem is called finite if it is feasible and the infimum of $J_{MV}(x,i,u(\cdot))$ is finite. Finally, an optimal portfolio to the above problem, if it ever exists, is called an efficient portfolio corresponding to $z$, and the corresponding $(\text{Var } X^{u}(T;x,i),z)$ is called an efficient point. The set of all the efficient points is called the efficient frontier.
\end{definition}

We point out that the formulation of Problem (MV) is consistent with the MV portfolio selection problems studied in \cite{Lim-2002-MV,Zhou-2003-MV,Zhang-Elliott-Siu-2012,Shen-2022-MV}, et al.
%\subsection{Feasibility}
We begin to solve this problem by considering the following system of linear ordinary differential equations (ODEs, for short):
\begin{equation}\label{psi}
  \left\{
  \begin{array}{l}
  \dot{\psi}(t,i)=-r(t,i)\psi(t,i)-\sum_{j=1}^{L}\pi_{ij}(t)\psi(t,j),\\[2mm]
  \psi(T,i)=1,\quad i\in\mathcal{S}.
  \end{array}
  \right.
\end{equation}
Then we have the following result.

\begin{proposition}\label{prop-feasibility}
  Let $\psi(\cdot,i),\text{ }i\in\mathcal{S}$, be the solutions of the coupled linear ODEs \eqref{psi}. Then the Problem (MV) is feasible for every $z\in\mathbb{R}$ if and only if
  \begin{equation}\label{feasibility-condition}
  \begin{array}{c}
  \delta:=\mathbb{E}\int_{0}^{T}\big|\psi(t,\alpha(t))B(t,\alpha(t))+\sum_{j=1}^{L}\lambda_{j}(t)\big[\psi(t,j)-\psi(t,\alpha(t))\big]\gamma_{j}(t,\alpha(t))\big|^{2}dt>0.
   \end{array}
  \end{equation}
\end{proposition}
\begin{proof}
\textbf{Sufficiency.}
  Let
  $$
  u_{0}(t)\equiv [0,0,\cdots,0]^{\top},\quad
 u(t)=\psi(t,\alpha(t-))B(t,\alpha(t-))+\sum_{j=1}^{L}\pi_{\alpha(t-)j}\psi(t,j)\gamma_{j}(t,\alpha(t-)).
  $$
  Then for any $\beta\in\mathbb{R}$, the associated expected terminal wealth to the admissible strategy $u^{\beta}(\cdot):=u_{0}(\cdot)+\beta u(\cdot)$ is given by
  $$
  \mathbb{E}[X^{u^{\beta}}(T;x,i)]=\mathbb{E} [X^{u_{0}}(T;x,i)]+\beta\mathbb{E}[X^{u}(T;0,i)],
  $$
  with
\begin{equation}\label{z0}
\mathbb{E} [X^{u_{0}}(T;x,i)]=\mathbb{E} [xe^{\int_{0}^{T}r(t,\alpha(t))dt}]\triangleq z^{0}.
\end{equation}
%and $X_{0}^{\pi}(\cdot)$ is the solution to the following SDE:
%\begin{equation}\label{wealth-0}
%\left\{
%\begin{array}{l}
%dX_{0}^{\pi}(t)=\big[r(t,\alpha(t))X_{0}^{\pi}(t)+B(t,\alpha(t))\pi(t)\big]dt+\sigma(t,\alpha(t))\pi(t)dW(t)
%+\sum_{j=1}^{L}\gamma_{j}(t,\alpha(t-))\pi(t)d\widetilde{N}_{j}(t),\\[2mm]
%X_{0}^{\pi}(0)=0,\quad \alpha(0)=i.
%\end{array}
%\right.
%\end{equation}
Applying the It\^o's rule to $\psi(t,\alpha(t))X^{u}(t;0,i)$ and taking the expectation, we have
\begin{equation}\label{p-1}
\begin{array}{l}
  \mathbb{E}[X^{u}(T;0,i)]=\mathbb{E}[\psi(T,\alpha(T))X^{u}(T;0,i)]\\[2mm]
 =\mathbb{E}\int_{0}^{T}\big\{\psi(t,\alpha(t))\big[r(t,\alpha(t))X^{u}(t;0,i)+B(t,\alpha(t))u(t)\big]\\
  \quad+\big[\dot{\psi}(t,\alpha(t))+\sum_{j=1}^{L}\lambda_{j}(t)(\psi(t,j)-\psi(t,\alpha(t)))\big]X^{u}(t;0,i)\\[2mm]
  \quad+\sum_{j=1}^{L}\lambda_{j}(t)(\psi(t,j)-\psi(t,\alpha(t)))\gamma_{j}(t,\alpha(t-))u(t)\big\}dt\\[2mm]
 =\mathbb{E}\int_{0}^{T}\big\{\big[\psi(t,\alpha(t))r(t,\alpha(t))
 +\dot{\psi}(t,\alpha(t))+\sum_{j=1}^{L}\lambda_{j}(t)(\psi(t,j)-\psi(t,\alpha(t)))\big]X^{u}(t;0,i)\\[2mm]
 \quad+\big[\psi(t,\alpha(t))B(t,\alpha(t))+\sum_{j=1}^{L}\lambda_{j}(t)(\psi(t,j)-\psi(t,\alpha(t)))\gamma_{j}(t,\alpha(t-))\big]u(t)\big\}dt\\[2mm]
  = \mathbb{E}\int_{0}^{T}\big|\psi(t,\alpha(t))B(t,\alpha(t))+\sum_{j=1}^{L}\lambda_{j}(t)(\psi(t,j)-\psi(t,\alpha(t)))\gamma_{j}(t,\alpha(t))\big|^{2}dt>0.
  \end{array}
\end{equation}
Consequently, for any given $z\in\mathbb{R}$, we complete the sufficiency proof by selecting $\beta=\frac{z-z^{0}}{\delta}$.

\textbf{Necessity.} Suppose that the condition \eqref{feasibility-condition} invalid. Then for any $i\in\mathcal{S}$, the following holds:
\begin{equation}\label{pp}
  \psi(t,i)B_{k}(t,i)+\sum_{j\neq i}^{L}\pi_{ij}\big[\psi(t,j)-\psi(t,i)\big]\gamma_{kj}(t,i)=0,\quad a.e. \text{ }t\in[0,T], \quad \forall k=1,2,\cdots,m.
\end{equation}
In this case, for any admissible portfolio $u(\cdot)$, we have
$$
\begin{aligned}
\mathbb{E}[X^{u}(T;x,i)]&=z^{0}+\mathbb{E}[X^{u}(T;0,i)]\\
&=z^{0}+\mathbb{E}\int_{0}^{T}\big[\psi(t,\alpha(t))B(t,\alpha(t))
+\sum_{j=1}^{L}\lambda_{j}(t)(\psi(t,j)-\psi(t,\alpha(t)))\gamma_{j}(t,\alpha(t-))\big]u(t)dt\\
&=z^{0}.
\end{aligned}
$$
Hence, Problem (MV) is feasible only if $z=z^{0}$, which contradicts the feasibility of Problem (MV) for every $z\in\mathbb{R}$. Consequently, we complete the necessity proof by contradiction.
\end{proof}

\begin{remark}\rm
  Based on the above analysis, we can find that the condition \eqref{feasibility-condition} is very mild, and it is invalid if and only if for any $i\in\mathcal{S}$, the condition \eqref{pp} holds. Clearly, this is impossible for complex and volatile financial markets.
\end{remark}

Notice that Problem (MV) is a constrained stochastic optimization problem, with an equality constraint. To solve this problem, we consider the following modified objective:
\begin{equation}\label{J1}
  J_{1}(x,i;u(\cdot),\theta)\triangleq \mathbb{E}\big[\big(X^{u}(T)-z\big)^{2}+2\theta\big(\mathbb{E}[X^{u}(T)]-z\big)\big]
  =\mathbb{E}\big[\big(X^{u}(T)+(\theta-z)\big)^{2}\big]-\theta^{2},
\end{equation}
where $\theta \in \mathbb{R}$ is the Lagrange multiplier. Then, by the Lagrangian duality theorem, we know that solving the Problem  (MV) is equivalent to solving the following max-min problem:
\begin{equation}\label{max-min}
 %\textbf{Problem 1: }
 \max_{\theta\in\mathbb{R}}\min_{u(\cdot)\in \mathcal{A}} \text{ }J_{1}(x,i;u(\cdot),\theta)
 % :=\max_{\theta\in\mathbb{R}}\min_{\pi(\cdot)\in \mathcal{A}}
 % \mathbb{E}\big[\big(X^{u}(T)+(\theta-z)\big)^{2}\big]-\theta^{2}.
\end{equation}
Consequently, to solve the above problem, we first need to solve the following unconstrained minimization problem:
\begin{equation}\label{min}
%\textbf{Problem 2: }
\min_{u(\cdot)\in \mathcal{A}} \text{ }J_{2}(x,i;u(\cdot),c)
  :=%\min_{\pi(\cdot)\in \mathcal{A}}
  \mathbb{E}\big[\big(X^{u}(T)+c\big)^{2}\big].
\end{equation}
Once we obtain the optimizer $u^{*}(\cdot;c)$ to problem \eqref{min} parameterized by $c\in\mathbb{R}$, we can further solve the problem \eqref{max-min} by considering
\begin{equation}\label{max}
\max_{\theta\in\mathbb{R}}\text{ }J_{3}(x,i;\theta)
:=%\max_{\theta\in\mathbb{R}}
J_{2}(x,i;u^{*}(\cdot;\theta-z),\theta-z)-\theta^{2}.
\end{equation}
Let $\theta^{*}$ be the maximizer of problem \eqref{max}. Then the efficient strategy for Problem (MV) is given by
\begin{equation}\label{efficient-strategy-theta}
 u^{*}(\cdot)=u^{*}(\cdot;\theta^{*}-z).
\end{equation}

%\subsection{Solution to the unconstrained problem}
Now, we return to solve the unconstrained problem \eqref{min}, which is a standard SLQ control problem for a general regime-switching jump diffusion model.
%with $A(s,\alpha(s))$ and $Q(s,\alpha(s))\equiv 0,\text{ }S(s,\alpha(s))\equiv 0,\text{ }R(s,\alpha(s))\equiv 0,\text{ }G(T,\alpha(T))\equiv 1,\text{ }g\equiv c$.
To adopt the derived results, we first need to verify that problem \eqref{min} is uniformly convex, that is to say, we need to verify that  there exists a constant $\mu>0$ such that
\begin{equation}\label{uniformly-convex-condition-exam}
  \mathbb{E} [X^{u}(T;0,i)^{2}]>\mu\mathbb{E}\int_{0}^{T}\big|u(t)\big|^{2}dt,\quad \forall u(\cdot)\in\mathcal{A},\quad \forall i\in\mathcal{S}.
\end{equation}
%where $X_{0}^{\pi}(\cdot)$ is the solution to \eqref{wealth-0}.

\begin{proposition}\label{prop-convex-exam}
  If the uniformly nondegenerate condition \eqref{uniformly-nondegenerate-condition} holds, then the following holds:
  $$\mathbb{E} [X^{u}(T;x,i)^{2}]>\mu\mathbb{E}\int_{0}^{T}\big|u(t)\big|^{2}dt,\quad \forall u(\cdot)\in\mathcal{A},\quad \forall (x,i)\in\mathbb{R}\times \mathcal{S}.$$
\end{proposition}
\begin{proof}
To obtain the estimation  \eqref{uniformly-convex-condition}, we need to review SDE \eqref{wealth} as a BSDE. For any given $u(\cdot)\in\mathcal{A}$ and $i\in\mathcal{S}$, let
$$
\mathcal{Y}(t)=X^{u}(t;x,i),\quad \mathcal{Z}(t)=\sigma(t,\alpha(t))u(t),\quad \Xi_{j}(t)=\gamma_{j}(t,\alpha(t-))u(t),\quad a.e.\text{ }t\in[0,T],\quad a.s.,
$$
and
$$
\begin{array}{c}
\mathcal{H}(t,i)=\sigma(t,i)^{\top}\sigma(t,i)+\sum_{j=1}^{L}\lambda_{j}(t)\gamma_{j}(t,i)^{\top}\gamma_{j}(t,i)\geq \delta I,
\quad a.e.\text{ }t\in[0,T].
\end{array}
$$
Then we have
$$
\begin{array}{c}
u(t)=\mathcal{H}(t,\alpha(t-))^{-1}\big[\sigma(t,\alpha(t-))^{\top}\mathcal{Z}(t)
+\sum_{j=1}^{L}\lambda_{j}(t)\gamma_{j}(t,\alpha(t-))^{\top}\Xi_{j}(t)\big],\quad a.e.\text{ }t\in[0,T],\quad a.s.,
\end{array}
$$
and $(\mathcal{Y}(\cdot),\mathcal{Z}(\cdot),\mathbf{\Xi}(\cdot))$ solves the following BSDE:
\begin{equation}\label{Y-BSDE}
\left\{
\begin{array}{l}
  d\mathcal{Y}(t)=\big\{B(t,\alpha(t))\mathcal{H}(t,\alpha(t))^{-1}\big[\sigma(t,\alpha(t))^{\top}\mathcal{Z}(t)+\sum_{j=1}^{L}\lambda_{j}(t) \gamma_{j}(t,\alpha(t))^{\top}\Xi_{j}(t)\big]\\[2mm]
   \qquad \qquad+r(t,\alpha(t))\mathcal{Y}(t)\big\}dt+\mathcal{Z}(t)dW(t)+\sum_{j=1}^{L}\Xi_{j}(t)d\widetilde{N}_{j}(t),\quad t\in [0,T],\\[2mm]
   \mathcal{Y}(T)=X^{u}(T;x,i).
\end{array}
\right.
\end{equation}
Based on the priori estimate of the BSDE (see, for example, Delong \cite[Lemma 3.1.1]{Delong-2013}), we can further obtain
\begin{equation}\label{estmation}
  \begin{array}{c}
  \mathbb{E}\Big[\sup_{t\in[0,T]}\mathcal{Y}(t)^{2}
  +\int_{0}^{T}\big(\mathcal{Z}(t)^{2}+\sum_{j=1}^{L}\lambda_{j}(t)\Xi_{j}(t)^{2}\big)dt\Big]\leq K\mathbb{E}[X^{u}(T;x,i)^{2}],\quad \text{for some } K>0.
  \end{array}
\end{equation}
On the other hand, note that
%$\sigma(\cdot,i)\in L^{\infty}(0,T;\mathbb{R}^{m})$ and $\gamma_{j}(\cdot,i)\in L^{\infty}(0,T;(-1,+\infty)^{m})$,
$$
\begin{array}{l}
 \quad \mathbb{E}\Big[\int_{0}^{T}\big(\mathcal{Z}(t)^{2}+\sum_{j=1}^{L}\lambda_{j}(t)\Xi_{j}(t)^{2}\big)dt\Big]\\[2mm]
=\mathbb{E}\int_{0}^{T}u(t)^{\top}\big(\sigma(t,\alpha(t))^{\top}\sigma(t,\alpha(t))
+\sum_{j=1}^{L}\lambda_{j}(t)\gamma_{j}(t,\alpha(t))^{\top}\gamma_{j}(t,\alpha(t))\big)u(t)dt\\[2mm]
\geq \delta \mathbb{E}\int_{0}^{T}|u(t)|^{2}dt.
\end{array}
$$
Hence, substituting the above estimation into \eqref{estmation}, we obtain the desired result.
\end{proof}

Having obtained the uniform convexity of problem \eqref{min}, we proceed with the study of its optimal control. First, the notation \eqref{notation-MLN} in this special case becomes:
\begin{equation}\label{notation-MLN-exam}
\left\{
  \begin{array}{l}
  \mathcal{M}(s,i;\mathbf{P})\triangleq 2P(s,i)r(s,i)+\sum_{j\neq i}^{L} \pi_{ij}(s)\big[P(s,j)-P(s,i)\big],\\[2mm]
  \mathcal{L}(s,i;\mathbf{P})\triangleq P(s,i)B(s,i)+\sum_{j\neq i}^{L} \pi_{ij}(s)\big[P(s,j)-P(s,i)\big]\gamma_{j}(s,i),\\[2mm]
   \mathcal{N}(s,i;\mathbf{P})\triangleq P(s,i)\sigma(s,i)^{\top}\sigma(s,i)+\sum_{j\neq i}^{L} \pi_{ij}(s)P(s,j)\gamma_{j}(s,i)^{\top}\gamma_{j}(s,i),\quad s\in[0,T],\text{ }i\in\mathcal{S}.
  \end{array}
  \right.
\end{equation}
By Theorem \ref{thm-CDRE-uniform-convex}, the CDREs:
\begin{equation}\label{CDREs-exam}
 \left\{
 \begin{aligned}
 -\dot{P}(s,i)&=\mathcal{M}(s,i;\mathbf{P})-\mathcal{L}(s,i;\mathbf{P})\mathcal{N}(s,i;\mathbf{P})^{-1}\mathcal{L}(s,i;\mathbf{P})^{\top},\quad s\in[0,T],\\
 P(T,i)&=1,\quad i\in\mathcal{S},
 \end{aligned}
 \right.
\end{equation}
admits a unique solution $\mathbf{P}(\cdot)\in\mathcal{D}(C(0,T;\mathbb{R}))$ such that $\mathcal{N}(\cdot,i;\mathbf{P})\gg 0$. In fact, by Proposition \ref{prop-convex-exam} and Corollary \ref{coro-closed-representation-0}, we can further obtain that $P(\cdot,i)\gg 0$ for all $i\in\mathcal{S}$.

On the other hand, the associated BSDE \eqref{eta} for problem \eqref{min} becomes:
\begin{equation}\label{eta-exam}
\left\{
\begin{array}{l}
d\eta(s)=\Big\{\big[\mathcal{L}(s,\alpha(s);\mathbf{P})\mathcal{N}(s,\alpha(s);\mathbf{P})^{-1}B(s,\alpha(s))^{\top}-r(s,\alpha(s))\big]\eta(s)\\[2mm]
\qquad\quad+\sum_{j=1}^{L}\lambda_{j}(s)\mathcal{L}(s,\alpha(s);\mathbf{P})\mathcal{N}(s,\alpha(s);\mathbf{P})^{-1}\gamma_{j}(s,\alpha(s))^{\top}z_{j}(s)\Big\}dt
+\sum_{j=1}^{L}z_{j}(s)d \widetilde{N}_{j}(s),\\[2mm]
\eta(T)=c.
\end{array}
\right.
\end{equation}
The following result provides a visual result for the solution to BSDE \eqref{eta}.
\begin{proposition}\label{prop-H}
Let $\mathbf{H}(\cdot)\in \mathcal{D}(C(0,T;\mathbb{R}))$ be the solution to linear differential equations:
\begin{equation}\label{H}
\left\{
\begin{array}{l}
\dot{H}(s,i)=\big\{\big[\mathcal{L}(s,i;\mathbf{P})\mathcal{N}(s,i;\mathbf{P})^{-1}B(s,i)^{\top}-r(s,i)\big]H(s,i)\\[2mm]
\qquad \qquad +\sum_{j=1}^{L}\pi_{ij}(s)\big[\mathcal{L}(s,i;\mathbf{P})\mathcal{N}(s,i;\mathbf{P})^{-1}\gamma_{j}(s,i)^{\top}-1\big]\big[H(s,j)-H(s,i)\big],\\[2mm]
H(T,i)=1,\quad i\in\mathcal{S}.
\end{array}
\right.
\end{equation}
Then the unique solution $(\eta(\cdot),\mathbf{z}(\cdot))\in \mathcal{S}_{\mathbb{F}}^{2}(0,T;\mathbb{R})\times\mathcal{D}(L_{\mathcal{P}}^{2}(0,T;\mathbb{R}))$ to BSDE \eqref{eta-exam} can be represented as
\begin{equation}\label{eta-representation}
  \eta(s)=cH(s,\alpha(s)),\quad z_{j}(s)=c\big[H(s,j)-H(s,\alpha(s-))\big],\quad j\in\mathcal{S},\quad a.e.\text{ }s\in[0,T],\quad a.s..
\end{equation}
\end{proposition}
\begin{proof}
  Clearly, by \eqref{H}, we have
$$
\left\{
\begin{array}{l}
d\big[cH(s,\alpha(s))\big]=\big\{\big[\mathcal{L}(s,\alpha(s);\mathbf{P})\mathcal{N}(s,\alpha(s);\mathbf{P})^{-1}B(s,\alpha(s))^{\top}-r(s,\alpha(s))\big]cH(s,\alpha(s))\\[2mm]
\qquad \qquad +\sum_{j=1}^{L}\lambda_{j}(s)\mathcal{L}(s,\alpha(s);\mathbf{P})\mathcal{N}(s,\alpha(s);\mathbf{P})^{-1}\gamma_{j}(s,\alpha(s))^{\top}\big[cH(s,j)-cH(s,\alpha(s))\big]ds\\[2mm]
\qquad \qquad
+\sum_{j=1}^{L} \big[cH(s,j)-cH(s,\alpha(s-))\big]d\widetilde{N}_{j}(s),\quad s\in[0,T],\\[2mm]
cH(T,\alpha(T))=c,\quad i\in\mathcal{S}.
\end{array}
\right.
$$
The desired result follows from the unique solvability of BSDE \eqref{eta-exam}.
\end{proof}

By Theorem \ref{thm-closed-representation}, we have the following result directly.
\begin{theorem}\label{thm-min-problem}
   Suppose that the condition \eqref{uniformly-nondegenerate-condition} holds. Let $\mathbf{P}(\cdot)\in\mathcal{D}(C(0,T;\mathbb{R}))$ and $\mathbf{H}(\cdot)\in\mathcal{D}(C(0,T;\mathbb{R}))$ be the solution to CDREs \eqref{CDREs-exam} and ODEs \eqref{H}, respectively.
  Then the problem \eqref{min}  admits a unique open-loop optimal control:
  \begin{equation}\label{optimal-control-min}
    u^{*}(s;c)=-\mathcal{N}(s,\alpha(s-);\mathbf{P})^{-1}\big[\mathcal{L}(s,\alpha(s-);\mathbf{P})^{\top}X^{*}(s-;c)
 +c\widetilde{\rho}(s,\alpha(s-))\big],\quad a.s.,\text{ }a.e.\text{ }s\in[t,T],
  \end{equation}
  where $X^{*}(\cdot;c)$ is the solution to SDE:
  \begin{equation}\label{wealth-optimal}
\left\{
\begin{array}{l}
dX^{*}(s;c)=\big\{-B(s,\alpha(s))\mathcal{N}(s,\alpha(s);\mathbf{P})^{-1}\big[\mathcal{L}(s,\alpha(s);\mathbf{P})^{\top}X^{*}(s;c)
+c\widetilde{\rho}(s,\alpha(s))\big]\\[2mm]
\quad+r(s,\alpha(s))X^{*}(s;c)\big\}ds-\sigma(s,\alpha(s))\mathcal{N}(s,\alpha(s);\mathbf{P})^{-1}\big[\mathcal{L}(s,\alpha(s);\mathbf{P})^{\top}X^{*}(s;c)
+c\widetilde{\rho}(s,\alpha(s))\big]dW(s)\\[2mm]
\quad-\sum_{j=1}^{L}\gamma_{j}(s,\alpha(s-))\mathcal{N}(s,\alpha(s-);\mathbf{P})^{-1}\big[\mathcal{L}(s,\alpha(s-);\mathbf{P})^{\top}X^{*}(s-;c)
+c\widetilde{\rho}(s,\alpha(s-))\big]d\widetilde{N}_{j}(s),\\[2mm]
X^{*}(0;c)=x,\quad \alpha(0)=i,
\end{array}
\right.
\end{equation}
 and
 \begin{equation}\label{widetilde-rho-exam}
 \begin{array}{l}
  \widetilde{\rho}(s,i)=B(s,i))^{\top}H(s,i)+\sum_{j\neq i}^{L}\pi_{ij}(s)\gamma_{j}(s,i)^{\top}\big[H(s,j)-H(s,i)\big],\quad i\in\mathcal{S},\quad  a.s., \quad a.e.\text{ }s\in[t,T].
  \end{array}
  \end{equation}
   In this case, the value function of problem \eqref{min} is given by
  \begin{equation}\label{value-SLQ-exam}
   \begin{array}{l}
    J_{2}(x,i;u^{*}(\cdot;c),c)=P(0,i)x^{2}+2cH(0,i)x-\kappa c^{2}+c^{2},
   \end{array}
  \end{equation}
  where
 \begin{equation}\label{kappa}
  \kappa:=\mathbb{E}\Big[\int_{0}^{T}
  \sum_{j=1}^{L}\lambda_{j}(s)\big[\big<\mathcal{N}(s,\alpha(s);\mathbf{P})^{-1}\widetilde{\rho}(s,\alpha(s)),\widetilde{\rho}(s,\alpha(s))\big> ds\Big].
  \end{equation}
\end{theorem}

\begin{remark}\rm
  If $\gamma_{kj}(\cdot,i)\equiv 0$ for all $k=1,2,\cdots,m$ and $i,\text{ }j\in\mathcal{S}$, then the corresponding CDREs \eqref{CDREs-exam} and ODEs \eqref{H} are respectively degenerated into
  \begin{equation}\label{CDREs-exam-2}
    \left\{
 \begin{array}{l}
 -\dot{P}(s,i)=\big[2r(s,i)-B(s,i)\big(\sigma(s,i)^{\top}\sigma(s,i)\big)^{-1}B(s,i)^{\top}\big]P(s,i)+\sum_{j=1}^{L}\pi_{ij}(s)P(s,j)
 ,\quad s\in[0,T],\\[2mm]
 P(T,i)=1,\quad i\in\mathcal{S},
 \end{array}
 \right.
  \end{equation}
  and
  \begin{equation}\label{H-2}
    \left\{
 \begin{array}{l}
 \dot{H}(s,i)=\big[B(s,i)\big(\sigma(s,i)^{\top}\sigma(s,i)\big)^{-1}B(s,i)^{\top}-r(s,i)\big]H(s,i)-\sum_{j=1}^{L}\pi_{ij}(s)H(s,j)\quad s\in[0,T],\\[2mm]
 H(T,i)=1,\quad i\in\mathcal{S}.
 \end{array}
 \right.
  \end{equation}
  Let $\widehat{H}(\cdot,i)\triangleq P(\cdot,i)^{-1}H(\cdot,i)$, $i\in\mathcal{S}$. Then it follows from \eqref{CDREs-exam-2}-\eqref{H-2} that
    \begin{equation}\label{H-3}
    \left\{
 \begin{array}{l}
 \dot{\widehat{H}}(s,i)=r(s,i)\widehat{H}(s,i)-\frac{1}{P(s,i)}\sum_{j=1}^{L}\pi_{ij}(s)P(s,j)\big[\widehat{H}(s,j)-\widehat{H}(s,i)\big]\quad s\in[0,T],\\[2mm]
 \widehat{H}(T,i)=1,\quad i\in\mathcal{S}.
 \end{array}
 \right.
  \end{equation}
Consequently, the optimal feedback control of problem \eqref{min} for this special case is given by
$$
u^{*}(s,x,i;c)=-\big[\sigma(s,i)^{\top}\sigma(s,i)\big]^{-1}B(s,i)^{\top}\big[x+c\widehat{H}(s,i)\big],
$$
which is consistent with the result provided in Zhou and Yin \cite{Zhou-2003-MV}.
\end{remark}

%\subsection{Efficient frontier}
%In this section,
In the following, we proceed to derive the efficient frontier for the original MV portfolio selection problem \eqref{MV}. Firstly, based on Theorem \ref{thm-min-problem}, we have
\begin{equation}\label{J3-star}
\begin{aligned}
 J_{3}(x,i;\theta)&=J_{2}(x,i;u^{*}(\cdot;\theta-z),\theta-z)-\theta^{2}\\
 &=P(0,i)x^{2}+2(\theta-z)H(0,i)x-\kappa (\theta-z)^{2}+(\theta-z)^{2}-\theta^{2}\\
 %&=-\kappa\big[\theta-\frac{1}{\kappa}\big(H(0,i)x+\kappa z-z\big)\big]^{2}
 %-\frac{1}{\kappa}\big[H(0,i)x+\kappa z-z\big]^{2}+P(0,i)x^{2}-2H(0,i)xz+(1-\kappa)z^{2}.
 &=-\kappa(\theta-z)^{2}+2\big[H(0,i)x-z\big](\theta-z)+P(0,i)x^{2}-z^{2}.
 \end{aligned}
\end{equation}
It follows from assumption \eqref{feasibility-condition} and Proposition \ref{prop-feasibility} that the Problem (MV) is feasible for any $z\in\mathbb{R}$. Additionally, one can easily verify that the Problem (MV) without the constraint $\mathbb{E}X^{u}(T;x,i)=z$ has a finite optimal value, which implies that the Problem (MV) is also finite for any $z\in\mathbb{R}$. Adopting the well-known duality theorem, we obtain the optimal value of Problem (MV) given by
\begin{equation}\label{J3-value}
 J_{MV}^{*}(x,i)=\sup_{\theta\in\mathbb{R}} J_{3}(x,i;\theta)>-\infty.
\end{equation}
%Noting that $J_{3}(x,i;\theta)$ is a quadratic function in $\theta-z$, it follows from the finiteness of the supremum value of this quadratic function that $\kappa\geq 0.$
It follows from $\mathcal{N}(\cdot; P,i)\gg 0$ for all $i\in\mathcal{S}$ and equation \eqref{kappa} that $\kappa\geq 0$. If $\kappa=0$, then equations \eqref{J3-star} and \eqref{J3-value} yield $z=H(0,i)x$ for all $z\in\mathbb{R}$, which is a contradiction. Hence, we must have $\kappa>0$.
Consequently, let $\theta^{*}$ be the maximizer of $J_{3}(x,i;\theta)$. Then it must satisfy
$$
\theta^{*}-z=\frac{1}{\kappa}\big[H(0,i)x-z\big],
$$
which implies that
\begin{equation}\label{theta-star}
\theta^{*}=\frac{1}{\kappa}\big[H(0,i)x-z\big]+z.
\end{equation}
In this case, the efficient portfolio corresponding to $z$ is
\begin{equation}\label{efficient-portfolio}
  u^{*}(s)=-\mathcal{N}(s,\alpha(s-);\mathbf{P})^{-1}\big[\mathcal{L}(s,\alpha(s-);\mathbf{P})^{\top}X^{*}(s-;\theta^{*}-z)
 +(\theta^{*}-z)\widetilde{\rho}(s,\alpha(s-))\big],\quad a.s.,\text{ }a.e.
\end{equation}
where $X^{*}(\cdot;\theta^{*}-z)$ defined in \eqref{wealth-optimal} with $c=\theta^{*}-z$ and $\widetilde{\rho}(\cdot,i)$ defined in \eqref{widetilde-rho-exam}. Furthermore, the optimal value of Problem (MV) is given by
\begin{equation}\label{MV-value}
\begin{aligned}
 Var X^{u^{*}}(T;x,i)&=-\kappa(\theta^{*}-z)^{2}+2\big[H(0,i)x-z\big](\theta^{*}-z)+P(0,i)x^{2}-z^{2}\\
 &=\frac{1}{\kappa}\big[H(0,i)x-z\big]^{2}+P(0,i)x^{2}-z^{2}.
 \end{aligned}
\end{equation}

To sum up, we have the main result of this section.

\begin{theorem}\label{thm-min-problem-2}
 Suppose the conditions \eqref{uniformly-nondegenerate-condition} and  \eqref{feasibility-condition} hold. Let $\mathbf{P}(\cdot)\in\mathcal{D}(C(0,T;\mathbb{R}))$ and $\mathbf{H}(\cdot)\in\mathcal{D}(C(0,T;\mathbb{R}))$ be the solution to CDREs \eqref{CDREs-exam} and ODEs \eqref{H}, respectively. Then the Problem (MV) admits a unique efficient portfolio for any $z\in\mathbb{R}$ and the corresponding efficient point is given by \eqref{efficient-portfolio}-\eqref{MV-value}.

\end{theorem}

%\bibliography{references}
%\bibliographystyle{siamplain}

\end{document}